\documentclass[final, 3p, times,review]{elsarticle}
\usepackage{lineno,hyperref}
\usepackage{amsthm}
\usepackage{amsmath,amssymb,amsfonts,amscd}
\usepackage{latexsym}
\usepackage{algorithmicx,algorithm}
\usepackage{subfigure}
\usepackage{epstopdf}
\usepackage{url}
\usepackage{xcolor}
\usepackage{float}
\usepackage{color}

\usepackage{geometry}
\newtheorem{theorem}{Theorem}[section]
\newtheorem{definition}[theorem]{Definition}

\newtheorem{proposition}[theorem]{Proposition}

\numberwithin{equation}{section}
\usepackage{latexsym}
\usepackage{url}
\usepackage{xcolor}
\definecolor{newcolor}{rgb}{.8,.349,.1}

\begin{document}
\begin{frontmatter}
\title{Variational Discretizations for Hamiltonian Systems}	
\author[1,2]{Yihan Shen}
\cortext[cor1]{Corresponding author at: LSEC, ICMSEC, Academy of Mathematics and Systems Science, Chinese Academy of Sciences, Beijing 100190, China}
\author[1,2]{Yajuan Sun\corref{cor1}}
\ead{sunyj@lsec.cc.ac.cn}
\address[1]{LSEC, ICMSEC, Academy of Mathematics and Systems Science, Chinese Academy of Sciences, Beijing 100190, China}
\address[2]{School of Mathematical Sciences, University of Chinese Academy of Sciences, Beijing 100049, China}

\begin{abstract}
   In this paper, we study the Lagrangian functions for a class of second-order differential systems arising from physics. For such systems, we present necessary and sufficient conditions for the existence of Lagrangian functions. Based on the variational principle and the splitting technique, we construct variational integrators and prove their equivalence to the composition of explicit symplectic methods. We apply the newly derived variational integrators to the Kepler problem and demonstrate their effectiveness in numerical simulations. Moreover, using the modified Lagrangian, we analyze the dynamical behavior of the numerical solutions in preserving the Laplace--Runge--Lenz (LRL) vector.
\end{abstract}
\end{frontmatter}

\section{Introduction}
Many physical systems, such as the classical Kepler problem and charged particle systems in electromagnetic fields, are modeled by second-order differential equations. These systems typically exhibit rich conservative properties and are described by formulations with geometric structures \cite{feng1995collected, feng2010symplectic, hairer2006geometric}, such as Hamiltonian systems, volume-preserving systems, Poisson systems, and integral-preserving systems. Therefore, the numerical solutions of these systems are designed to preserve the geometric structures of the system. Such methods, which can preserve these geometric structures, have been shown to be capable of performing numerical simulations over long periods.

Variational structures are fundamental in describing systems within a variational framework. Using inverse variational techniques \cite{bampi1982inverse, santilli1978foundations, tonti1969variational}, the conditions for systems to admit a variational description can be established. Solving these systems corresponds to finding the minimum of an action functional derived from a Lagrangian. When the Lagrangian is regular, the variational formulation can be connected to the Hamiltonian description through Legendre transformations. By discretizing the Lagrangian, the discrete Euler–Lagrange equations provide variational discretizations that preserve the systems' variational structure \cite{marsden2001discrete}. Applications of variational methods include the motion of charged particles in electromagnetic fields \cite{hairer2020longterm, hairer2023leapfrog, qin2008variational, xiao2019explicit} and kinetic models in plasma physics \cite{kraus2015variational, kraus2016variational, kraus2021variational, xiao2018structure}.

The Kepler problem models the motion of a point mass in a classical gravitational potential \cite{arnold1989mathematical}. For negative total energy, the motion follows an elliptic orbit. This system is super-integrable, with conserved quantities: energy, angular momentum, and the Laplace--Runge--Lenz (LRL) vector. The LRL vector is often referred to as a "hidden conserved quantity" \cite{goldstein2011}, as it determines the orientation and shape of the Keplerian orbit. Preserving these conserved quantities numerically is crucial for maintaining the stability of simulated orbits and preventing long-term drift. Geometric integrators have demonstrated superior performance in long-term numerical simulations of the Kepler problem (e.g., \cite{chin1997symplectic, kozlov2007conservative, minesaki2002new, minesaki2004new}).

In this paper, we construct the first order and second order variational integrators for the Kepler problem by splitting its potential function. For comparison, we also use the symplectic Euler and St\"{o}rmer--Verlet methods. To analyze the dynamical behavior of these integrators, we derive their modified differential equations \cite{mclachlan2020backward, moan2006modified, vermeeren2017modified, oliver2024new}. Since variational integrators preserve the variational structure, the corresponding modified equations yield a perturbed but still conservative Kepler problem. Due to the perturbation, the modified equations are typically not super-integrable, although the LRL vector remains bounded. For the two-dimensional Kepler problem, the numerical orbits computed by the proposed integrators are shown to satisfy a perturbed Kepler problem with only two formal conservation laws. We demonstrate that the newly derived variational integrators are equivalent to compositions of explicit symplectic methods, as confirmed by efficient numerical experiments. Numerical results further show that the proposed integrators perform better than the symplectic Euler and St\"{o}rmer--Verlet methods. Using Noether's theory, along with the concepts of generalized vector fields, characteristics of vector fields, and conservation laws \cite{noether1918invariante, olver1993applications}, we analyze the errors of the LRL vector in our variational integrators. We conclude that the errors of the newly derived variational methods are smaller than those of the symplectic Euler and St\"{o}rmer--Verlet methods, confirming the accuracy of the numerical results.

The outline of this paper is as follows. Section 2 establishes the conditions under which second order systems can be described in a variational framework. In this section, for the two-dimensional Kepler problem, we also derive the variational symmetries corresponding to its three conservation laws. In Section 3, we construct the variational discretizations for both the Kepler problem and the relativistic Kepler problem, and prove the equivalence of the newly constructed variational integrators with compositions of explicit symplectic methods. In Section 4, we derive the modified Lagrangian and modified equations for second order differential equations, and analyze the errors of the proposed variational integrators in preserving conserved quantities. In Section 5, we apply the newly derived variational integrators to the Kepler problem and demonstrate the efficiency of the new methods. Finally, Section 6 concludes the paper.
	
\section{Variational Principle in Lagrangian Formalism}
\begin{definition}\label{def:varder}
    Denote $N$ as an operator on the function space ${\mathcal F}(\mathbf{R}^{n})$. Its variational derivative at $x$ with respect to $\delta x$ is defined as
    \[
        N_x \delta x = \left.\frac{d}{d\varepsilon} \right|_{\varepsilon=0} N(x + \varepsilon \delta x),
    \]
    where $x, \delta x \in {\mathcal F}(\mathbf{R}^{n})$.
\end{definition}
    
\begin{definition}\label{def:N*}
    Consider the operator $N$. The operator $N^{*}$ is called its adjoint if
    \[
        \langle N^{*} x, y \rangle = \langle x, N y \rangle, \quad \forall x, y \in {\mathcal F}(\mathbf{R}^{n}),
    \]
    where $\langle \cdot, \cdot \rangle$ denotes the inner product on ${\mathcal F}(\mathbf{R}^{n})$. Moreover, the operator $N$ is called self-adjoint if $N = N^{*}$.
\end{definition}
	
\begin{theorem}\label{Thm:Vainberg}
    Let
    \[
        N[x] := F(t, x, \dot{x}, \dots, x^{(m)}) = 0
    \]
    be a differential system, where $[x]$ contains $t$, $x$, and all the derivatives of $x$. The differential equation $N[x] = 0$ can be derived from a variational principle with the Lagrangian $L$ if the variational derivative of $N$ is self-adjoint, i.e., $N_x = N_x^*$. Furthermore, the Lagrangian $L$ can be obtained by
    \begin{equation}\label{eq:Lagrangian}
        L[x] = \int_{0}^{1} x N[\lambda x] \, d\lambda.
    \end{equation}
\end{theorem}

As follows, we consider second order differential equations of the form
\begin{align}\label{eq:DE_VP}
    \frac{d}{dt}\left(M(x, \dot{x}) \dot{x}\right) = f(t, x, \dot{x}), \quad x \in \mathbf{R}^{n},
\end{align}
where $M(x, \dot{x})$ is assumed to be an $n \times n$ symmetric matrix.

\begin{theorem}\label{Thm:2ndOde from VP}
    The system \eqref{eq:DE_VP} can be derived from a variational principle if the following conditions hold:
    \begin{subequations}\label{eq:varcond}
        \begin{align}
            &   \sum_{k=1}^{n} \frac{\partial M_{ik}}{\partial \dot{x}_{j}} \dot{x}_k = \sum_{k=1}^{n} \frac{\partial M_{jk}}{\partial \dot{x}_{i}} \dot{x}_k, \\
            &   \sum_{k=1}^{n}  \left(\frac{\partial M_{ik}}{\partial x_{j}} + \frac{\partial M_{jk}}{\partial x_{i}}\right) \dot{x}_k = \frac{\partial f_i}{\partial \dot{x}_j} + \frac{\partial f_j}{\partial \dot{x}_i}, \\
            &  \sum_{k=1}^{n}  \frac{d}{dt}\left(\frac{\partial M_{ik}}{\partial x_j} \dot{x}_k\right) = \frac{\partial f_i}{\partial x_j} - \frac{\partial f_j}{\partial x_i} + \frac{d}{dt}\frac{\partial f_j}{\partial \dot{x}_i}
        \end{align}
    \end{subequations}
    hold for all $i, j = 1, \dots, n$.
\end{theorem}
	
\begin{proof}
Let $N(x) = \frac{d}{dt} \left(M(x, \dot{x}) \dot{x}\right) - f(t, x, \dot{x})$. By Definition \ref{def:varder}, its variational derivative at $x$ is
\begin{align}\label{eq:Nx}
    N_x \delta x = \frac{d}{dt}\left[M \delta \dot{x} + \sum_{k=1}^{n} \left(\frac{\partial M}{\partial \dot{x}_k} \delta \dot{x}_k + \frac{\partial M}{\partial x_k} \delta x_k\right) \dot{x}\right] - \frac{\partial f}{\partial x} \delta x - \frac{\partial f}{\partial \dot{x}} \delta \dot{x}.
\end{align}
From Definition \ref{def:N*}, the adjoint of $N_x$ is
\begin{align}\label{eq:Nu}
    \langle N_x^* \delta x, \delta v \rangle = \langle \delta x, N_x \delta v \rangle = \left\langle
    \frac{d}{dt}\left[M \delta \dot{v} + \sum_{k=1}^{n}\left(\frac{\partial M}{\partial \dot{x}_k} \delta \dot{v}_k + \frac{\partial M}{\partial x_k} \delta v_k \right) \dot{x}\right]
    - \frac{\partial f}{\partial x} \delta v - \frac{\partial f}{\partial \dot{x}} \delta \dot{v}, \delta x \right\rangle.
\end{align}
Using integration by parts, we obtain
\begin{align}
    \langle N_x^* \delta x, \delta v \rangle = \left\langle
    \frac{d}{dt}\left[\frac{\partial}{\partial \dot{x}}\left(\dot{x}^{\top} M^{\top} \delta \dot{x}\right)\right]
    - \frac{\partial}{\partial x}\left[\dot{x}^{\top} M \delta \dot{x}\right]
    - \left(\frac{\partial f}{\partial x}\right)^{\top} \delta x
    + \frac{d}{dt}\left[\left(\frac{\partial f}{\partial \dot{x}}\right)^{\top} \delta x\right], \delta v
    \right\rangle.
\end{align}
Thus, we have
\begin{align}\label{eq:Nu*}
    N_x^* \delta x = \frac{d}{dt}\left[
    \frac{\partial}{\partial \dot{x}}\left(\dot{x}^{\top} M^{\top} \delta \dot{x}\right)
    \right]
    - \frac{\partial}{\partial x}\left[\dot{x}^{\top} M \delta \dot{x}\right]
    - \left(\frac{\partial f}{\partial x}\right)^{\top} \delta x
    + \frac{d}{dt}\left[\left(\frac{\partial f}{\partial \dot{x}}\right)^{\top} \delta x\right].
\end{align}
According to Theorem \ref{Thm:Vainberg}, the system can be derived from a variational principle if $N_x \delta x = N_x^* \delta x$. Comparing \eqref{eq:Nx} and \eqref{eq:Nu*} leads to
\begin{align*}
    &\left[M_{ij} +  \sum_{k=1}^{n} \frac{\partial M_{ik}}{\partial \dot{x}_{j}} \dot{x}_k - M_{ji} -  \sum_{k=1}^{n} \frac{\partial M_{jk}}{\partial \dot{x}_{i}} \dot{x}_k\right] \delta \ddot{x}_{j} \\
    +&\left[\frac{d}{dt}\left(M_{ij} +  \sum_{k=1}^{n} \frac{\partial M_{ik}}{\partial \dot{x}_{j}} \dot{x}_k - M_{ji} - \sum_{k=1}^{n} \frac{\partial M_{jk}}{\partial \dot{x}_{i}} \dot{x}_k\right)
    + \sum_{k=1}^{n}  \left(\frac{\partial M_{ik}}{\partial x_{j}} + \frac{\partial M_{jk}}{\partial x_{i}}\right) \dot{x}_k - \frac{\partial f_{i}}{\partial \dot{x}_{j}} - \frac{\partial f_{j}}{\partial \dot{x}_{i}}\right] \delta \dot{x}_{j} \\
    +& \left[\frac{d}{dt} \sum_{k=1}^{n} \frac{\partial M_{ik}}{\partial x_{j}} \dot{x}_k - \frac{\partial f_{i}}{\partial x_{j}} + \frac{\partial f_{j}}{\partial x_{i}} - \frac{d}{dt}\frac{\partial f_{j}}{\partial \dot{x}_{i}}\right] \delta x_{j} = 0.
\end{align*}
Since $M$ is symmetric, this equality implies exactly that
\[
\begin{aligned}
    & \sum_{k=1}^{n} \left(\frac{\partial M_{ik}}{\partial \dot{x}_{j}} \dot{x}_k - \frac{\partial M_{jk}}{\partial \dot{x}_{i}} \dot{x}_k \right) = 0, \\
    & \sum_{k=1}^{n} \left(\frac{\partial M_{ik}}{\partial x_{j}} + \frac{\partial M_{jk}}{\partial x_{i}}\right) \dot{x}_k - \frac{\partial f_{i}}{\partial \dot{x}_{j}} - \frac{\partial f_{j}}{\partial \dot{x}_{i}} = 0, \\
    & \frac{d}{dt}\left( \sum_{k=1}^{n} \frac{\partial M_{ik}}{\partial x_{j}} \dot{x}_k \right) - \frac{\partial f_{i}}{\partial x_{j}} + \frac{\partial f_{j}}{\partial x_{i}} - \frac{d}{dt} \frac{\partial f_{j}}{\partial \dot{x}_{i}} = 0.
\end{aligned}
\]
This completes the proof of the theorem.
\end{proof}
	
\begin{theorem}\label{Thm:M=M(dot x)}
    Suppose $M = M(\dot{x})$ and $f(t,x,\dot{x}) = A(t,x)\dot{x} + \varphi(t,x)$, then the system \eqref{eq:DE_VP} can be derived from a variational principle if the following conditions hold:
    \begin{subequations}\label{eq:varcond1}
        \begin{align}
            & \sum\limits_{k=1}^{n} \frac{\partial M_{ik}(\dot{x})}{\partial \dot{x}_j} \dot{x}_k = \sum\limits_{k=1}^{n} \frac{\partial M_{jk}(\dot{x})}{\partial \dot{x}_i} \dot{x}_k, \quad \forall i,j = 1, \dots, n, \\
            & A(t,x) = - A^{\top}(t,x), \\
            & \frac{\partial A_{jk}}{\partial x_i} + \frac{\partial A_{ki}}{\partial x_j} + \frac{\partial A_{ij}}{\partial x_k} = 0, \quad \forall i,j,k = 1, \dots, n, \\
            & \frac{\partial \varphi_i}{\partial x_j} - \frac{\partial \varphi_j}{\partial x_i} = \frac{\partial A_{ij}}{\partial t}, \quad \forall i,j = 1, \dots, n. \label{eq:M(dot x)3}
        \end{align}
    \end{subequations}
\end{theorem}

\begin{proof}
    Under the assumption of the theorem, substitute $M = M(\dot{x})$ and $f(t,x,\dot{x}) = A(t,x)\dot{x} + \varphi(t,x)$ into conditions \eqref{eq:varcond}, it is clear that the terms about the partial derivatives of $M$ with respect to $x$ vanish. The equalities (\ref{eq:varcond1}a) and (\ref{eq:varcond1}b) follow from (\ref{eq:varcond}a) and (\ref{eq:varcond}b). The condition (\ref{eq:varcond}c) for $i,j = 1, \dots, n$ is reduced to
    \[
    \sum\limits_{k=1}^{n}\left( \frac{\partial A_{jk}}{\partial x_i} + \frac{\partial A_{ki}}{\partial x_j} + \frac{\partial A_{ij}}{\partial x_k} \right) \dot{x}_k = \frac{\partial A_{ji}}{\partial t} + \frac{\partial \varphi_i}{\partial x_j} - \frac{\partial \varphi_j}{\partial x_i}.
    \]
    In the above equality, the right-hand side is independent of $\dot{x}$, thus the coefficients of the term $\dot{x}$ must vanish. This leads to equalities (\ref{eq:varcond1}c) and (\ref{eq:varcond1}d).
\end{proof}

Especially, if $M$ is a constant symmetric matrix, by Theorem \ref{Thm:2ndOde from VP}, the conditions for the variational description of system \eqref{eq:DE_VP} only depend on $f$.

\begin{theorem}\label{Thm:M}
    Suppose $M$ is a constant symmetric matrix. Then the system \eqref{eq:DE_VP} can be derived from a variational principle if the following conditions are satisfied:
    \begin{subequations}\label{eq:varcond2}
        \begin{align}
            & \frac{\partial f_i}{\partial \dot{x}_j} + \frac{\partial f_j}{\partial \dot{x}_i} = 0, \\
            & \frac{\partial f_i}{\partial x_j} - \frac{\partial f_j}{\partial x_i} + \frac{d}{dt} \frac{\partial f_i}{\partial \dot{x}_j} = 0.
        \end{align}
    \end{subequations}
\end{theorem}
	
	\noindent{\bf Kepler problem.} The Kepler problem describes the motion of two bodies attracting each other. If one body is placed at the origin, and $x$ represents the position of the other body, their dynamics are governed by the following system:
\begin{equation}\label{eq:Kepler}
    \ddot{x} = -\nabla \phi(x) = -\frac{x}{|x|^3}, \quad x \in \mathbf{R}^{n},
\end{equation}
where $\phi(x) = -\frac{1}{|x|}$. It is obvious that the Kepler problem \eqref{eq:Kepler} satisfies Theorem \ref{Thm:M}, and its Lagrangian is given by
\begin{align}\label{eq:Lagrangian_Kepler}
    L(x, \dot{x}) = \frac{1}{2} \dot{x}^\top \dot{x} + \frac{1}{|x|}.
\end{align}

The Kepler problem, which has a rich set of conservative laws, is a super-integrable system. For the $n$-dimensional Kepler problem \eqref{eq:Kepler}, we list the following conservation laws.

\begin{proposition}\label{Constant_Kepler}
    Consider the Kepler system \eqref{eq:Kepler}. The system can preserve the following conservative quantities:
    \begin{itemize}
        \item \textbf{Energy}
        \[
        H(x, \dot{x}) = \frac{1}{2} \dot{x}^\top \dot{x} - \frac{1}{|x|}.
        \]
        \item \textbf{Angular Momentum}
        \[
        \mathbf{m}(x, \dot{x}) = (m_{ij}(x, \dot{x}))_{n \times n},
        \]
        where $m_{ij}(x, \dot{x}) = x_{i} \dot{x}_{j} - x_{j} \dot{x}_{i}, \quad 1 \leq i < j \leq n.$
        When $n = 2$, $m = x_1 \dot{x}_2 - x_2 \dot{x}_1$; when $n = 3$, $\mathbf{m}$ can be expressed in compact form as
        \[
        \mathbf{m} = \dot{x} \times x.
        \]
        \item \textbf{Laplace--Runge--Lenz (LRL) Vector}
        \[
        \mathbf{A}(x, \dot{x}) = (A_i(x, \dot{x}))_{n \times 1},
        \]
        where $A_{i}(x, \dot{x}) = x_{i} |\dot{x}|^2 - \dot{x}_{i} (x \cdot \dot{x}) - \frac{x_{i}}{|x|}, \quad i = 1, \dots, n.$
        When $n = 3$, in compact form it reads
        \[
        \mathbf{A}(x, \dot{x}) = \dot{x} \times \mathbf{m} - \frac{x}{|x|}.
        \]
    \end{itemize}
\end{proposition}

In the following, we mainly focus on the 2-dimensional Kepler problem unless explicitly stated. Using the polar coordinates $x_1 = r\cos\theta$ and $x_2 = r\sin\theta$, Kepler's three laws (for more details, refer to \cite{goldstein2011}) can be described in the following proposition.

\begin{proposition}[Kepler's Three Laws]\label{Thm:Kepler_3}
    For negative energy, the orbits of the Kepler problem \eqref{eq:Kepler} are ellipses. Let $a$ and $b$ represent the semi-major and semi-minor axes of the orbit, and $e = \sqrt{1 - b^2/a^2}$ be the eccentricity. Then, we have:
    \begin{itemize}
        \item The inverse of the radial distance is given by
        \[
        \frac{1}{r} = \left(1 + e\cos\theta\right) \frac{a}{b^2};
        \]
        \item The areal velocity is conserved, expressed as
        \[
        r^2 \frac{d\theta}{dt} = \frac{2\pi ab}{T};
        \]
        \item The period of the orbit is
        \[
        T = 2\pi a^{3/2}.
        \]
    \end{itemize}
\end{proposition}

As follows, we study the relationship between conservation laws of a given system and their corresponding symmetries. In order to do that, we first need to introduce the generalized vector field and variational symmetry \cite{olver1993applications}.

\begin{definition}
    A generalized vector field $\mathbf{v}$ is expressed as
    \[
    \mathbf{v} = \sum\limits_{i=1}^{n} v_{i}[x] \frac{\partial}{\partial x_{i}},
    \]
    where $[x]$ includes $t$, $x$, and all the derivatives of $x$ with respect to $t$. The vector $v = \left(v_1, \dots, v_{n}\right)$ with smooth functions $v_{i}$ is called the characteristic of $\mathbf{v}$. Its infinite prolongation is
    \[
    \operatorname{pr} \mathbf{v} = \mathbf{v} + \sum\limits_{i=1}^{n} \sum\limits_{J=1} \frac{d v_{i}}{d t} \frac{\partial}{\partial {x}_{j}^{i}},
    \]
    where $x_{j}^{i} = \frac{\partial^j x_{i}}{\partial t^j}$.
\end{definition}

\begin{definition}\label{Var_sym}
    A generalized vector field $\mathbf{v}$ is called a variational symmetry of the action functional
    \[
    S[x] = \int L[x] \, dt,
    \]
    if there exists a function $K[x]$ such that
    \[
    \operatorname{pr} \mathbf{v}(L) = \frac{d}{dt} K[x]
    \]
    for all $t$ and $x$.
\end{definition}

Consider the system of $m$-dimensional differential equations $N_i[x] = 0$, $i = 1, \dots, m$. For the conservation law $\frac{d P}{dt} = 0$ of the system, there exists a vector function $Q[x]$ such that
\[
\frac{d P}{dt} = Q[x] N[x],
\]
where $Q$ is called the characteristic of the conservation law $\frac{d P}{dt} = 0$ \cite{olver1993applications}.

\begin{theorem}[Noether's theorem]\label{Thm:Noether}
    A generalized vector field $\mathbf{v}$ is a variational symmetry of $S(x) = \int L[x] \, dt$ if and only if its characteristic $Q$ is also the characteristic of a conservation law $\frac{d P}{dt} = 0$ for the Euler--Lagrange equations.
\end{theorem}

Based on Noether's theorem, for the two-dimensional Kepler problem \eqref{eq:Kepler}, we derive the variational symmetries corresponding to the three conservative quantities expressed in Proposition \ref{Constant_Kepler}.

\begin{proposition}\label{GVF_Kepler}
    The variational symmetries associated with the energy, angular momentum, and the LRL vector are as follows:
    \begin{itemize}
        \item $\mathbf{v}_{H} = \dot{x}_1 \frac{\partial}{\partial x_1} + \dot{x}_2 \frac{\partial}{\partial x_2},$
        \item $\mathbf{v}_m = - x_2 \frac{\partial}{\partial x_1} + x_1 \frac{\partial}{\partial x_2},$
        \item $\mathbf{v}_{A_1} = - x_2 \dot{x}_2 \frac{\partial}{\partial x_1} + \left(2 x_1 \dot{x}_2 - \dot{x}_1 x_2\right) \frac{\partial}{\partial x_2}$, $\mathbf{v}_{A_2} = \left(2 x_2 \dot{x}_1 - x_1 \dot{x}_2\right) \frac{\partial}{\partial x_1} - x_1 \dot{x}_1 \frac{\partial}{\partial x_2}.$
    \end{itemize}
\end{proposition}
	
\begin{proof}
    Taking the derivative of the energy $H$ with respect to $t$ gives
    \begin{align}\label{eq:hamsym}
        \frac{dH}{dt} = \dot{x}_1 \left( \frac{x_1}{|x|^3} + \ddot{x}_1 \right) + \dot{x}_2 \left( \frac{x_2}{|x|^3} + \ddot{x}_2 \right).
    \end{align}
    Denote $N[x] := \left( \frac{x_1}{|x|^3} + \ddot{x}_1, \frac{x_2}{|x|^3} + \ddot{x}_2 \right)^\top.$
    It follows from \eqref{eq:hamsym} that
    \begin{align}
        \frac{dH}{dt} = (\dot{x}_1, \dot{x}_2) N[x].
    \end{align}
    It is clear that $Q = (\dot{x}_1, \dot{x}_2)$ is the characteristic of $H$. The other variational symmetries can be obtained in a similar way.
\end{proof}
	
Denote the Euler operator as 
\begin{align}\label{eq:EL}
    \operatorname{EL}(L) = \frac{d}{dt} \frac{\partial L}{\partial \dot{x}} - \frac{\partial L}{\partial x}.
\end{align}
As follows, we analyze the effects of system perturbations on conserved quantities.
	
\begin{proposition}\label{PerGVF}
Let $L$ be a given Lagrangian and $\mathbf{v}$ be its variational symmetry.
Denote $P$ as the conserved quantity associated with the Euler--Lagrange equation $\operatorname{EL}(L) = 0$. Consider the perturbed Lagrangian $\widetilde{L} = L + \varepsilon \overline{L}$, then the time evolution of $P$ along the solution trajectory of the perturbed system $\operatorname{EL}(\widetilde{L}) = 0$ is given by
\[
    \frac{dP}{dt} = - \varepsilon \left\langle \operatorname{EL}(\overline{L}), \mathbf{v} \right\rangle,
\]
where $\langle \cdot, \cdot \rangle$ denotes the inner product.
\end{proposition}

\begin{proof}
By Noether's theorem (Theorem \ref{Thm:Noether}), it is known that the vector field $\mathbf{v}$ yields
    \begin{align}\label{eq:Noe}
        \frac{dP}{dt} = \left\langle \operatorname{EL}(L), \mathbf{v} \right\rangle.
    \end{align}
According to Definition \eqref{eq:EL}, for the solution of the perturbed system $\operatorname{EL}(\widetilde{L}) = 0$, we have
\[
    \left\langle \operatorname{EL}(L), \mathbf{v} \right\rangle + \varepsilon \left\langle \operatorname{EL}(\overline{L}), \mathbf{v} \right\rangle = 0.
\]
Thus, from \eqref{eq:Noe}, it follows that
\[
    \frac{dP}{dt} = - \varepsilon \left\langle \operatorname{EL}(\overline{L}), \mathbf{v} \right\rangle.
\]
\end{proof}
	
\section{Discrete Lagrangian mechanics}
Suppose that the time is discretized uniformly by a constant time-step $h$. Denote $x_{n}$ as the generalized coordinates at $t_{n} = t_{0} + nh$. The discrete Lagrangian $\mathbb{L}$ introduced in \cite{marsden2001discrete} can be understood as an approximation of the integral of the continuous Lagrangian over $\left[t_{n},t_{n+1}\right]$
\[
    h\mathbb{L}\left(x_{n},x_{n+1},h\right) \approx \int_{t_{n}}^{t_{n+1}} L(x,\dot{x})\,dt.
\]
The discrete action functional $\mathbb{S}$ is defined as a sum of the discrete Lagrangians indexed by time
\begin{align}\label{eq:disS}
    \mathbb{S}\left(\left\{x_{n}\right\}_{n=0}^{\mathrm{N}}\right) = h\sum_{n=0}^{\mathrm{N}-1} \mathbb{L}\left(x_{n},x_{n+1},h\right).
\end{align}
The discrete Hamilton's principle is to find the trajectories $\left\{ x_{n}\right\}_{n=0}^{\mathrm{N}}$ extremizing the discrete action functional. This gives the discrete Euler–Lagrange equations:
\[
    \partial_1 \mathbb{L}(x_{n},x_{n+1},h) + \partial_2\mathbb{L}(x_{n-1},x_{n},h) = 0.
\]

\begin{definition}\label{def:disLeg}
    Define the discrete Legendre transforms $\mathcal{F}_{\mathbb{L},h}^{-}$ and $\mathcal{F}_{\mathbb{L},h}^{+}$ as
    \begin{align*}
        \mathcal{F}_{\mathbb{L},h}^{-}:(x_{n},x_{n+1}) &\rightarrow (x_{n}, p_{n}) = \left(x_{n}, -h \partial_{1}\mathbb{L}(x_{n},x_{n+1},h)\right), \\
        \mathcal{F}_{\mathbb{L},h}^{+}:(x_{n},x_{n+1}) &\rightarrow (x_{n+1}, p_{n+1}) = \left(x_{n+1}, h \partial_{2}\mathbb{L}(x_{n},x_{n+1},h)\right).
    \end{align*}
    The discrete Lagrangian $\mathbb{L}(x_{n}, x_{n+1}, h)$ is called regular if ${\rm det}\left(\frac{\partial^2 \mathbb{L}(x_{n}, x_{n+1}, h)}{\partial x_{n} \partial x_{n+1}}\right)\neq 0$.
\end{definition}

Consider the second order differential system $\ddot{x} = -\nabla \phi(x)$. With the Lagrangian $L(x, \dot{x})=\frac{1}{2} \dot{x}^{\top} {\dot x}-\phi(x)$, the first order discrete Lagrangian can be taken as
\begin{align}\label{eq:disLag_symE}
    \mathbb{L}_{1}(x_{n}, x_{n+1}, h) = \frac{1}{2}\frac{\left( x_{n+1} - x_{n}\right)^{\top}\left( x_{n+1} - x_{n}\right)}{h^2} - \phi(x_{n}),
\end{align}
and the second order discrete Lagrangian can be chosen as
\begin{align}\label{eq:disLag_SV}
    \mathbb{L}_{2}(x_{n}, x_{n+1}, h) = \frac{1}{2}\frac{\left( x_{n+1} - x_{n}\right)^{\top}\left( x_{n+1} - x_{n}\right)}{h^2} - \frac{1}{2}\left[\phi(x_{n}) + \phi(x_{n+1})\right].
\end{align}
Both the discrete Lagrangians provide the same second order variational integrator via the discrete Euler--Lagrange equation
\[
    \frac{x_{n+1} - 2x_{n} + x_{n-1}}{h^2} = -\nabla \phi(x_{n}).
\]
This implies that the two discrete Lagrangians are \textbf{weakly equivalent} as described in \cite{marsden2001discrete}.

Introducing the discrete Legendre transform as defined in Definition \ref{def:disLeg}, the discrete Hamiltonian maps corresponding to \eqref{eq:disLag_symE} and \eqref{eq:disLag_SV} are the symplectic Euler and Störmer--Verlet methods, respectively.

\begin{definition}[The adjoint of the Lagrangian]\label{def:adj_Langrangian}
    Given $\mathbb{L}$, the adjoint of $\mathbb{L}$, denoted by $\mathbb{L}^*$, is defined as
    \[
        \mathbb{L}^*(x_{n},x_{n+1},h) = \mathbb{L}(x_{n+1},x_{n},-h).
    \]
\end{definition}
	
Consider the system
\begin{equation}\label{eq:Eq1}
    \ddot{x} = -\nabla\phi (x), \quad x \in \mathbf{R}^{\mathrm{N}}.
\end{equation}
It can be derived from the variational principle with Lagrangian $L(x, \dot{x})=\frac{1}{2} \dot{x}^{\top} \dot{x} - \phi(x)$. It also has the Hamiltonian form with $H(x, p) = \frac{1}{2} p^\top p + \phi(x)$, where $p = \dot{x}$.

Split $\phi(x)$ as $\phi(x) = \sum\limits_{i=1}^{\mathrm{N}} \phi^{[i]}(x)$, then the Hamiltonian can be decomposed as
\[
H(x, p) = \sum\limits_{i=1}^{\mathrm{N}} \frac{1}{2} p_{i}^2 + \phi^{[i]}(x) := \sum\limits_{i=1}^{\mathrm{N}} H^{[i]}(x, p).
\]
We can construct a symplectic numerical method for system \eqref{eq:Eq1} by composition, that is
\begin{align}\label{eq:Hcomposition}
    \Phi_{h} = \Phi_{H_{N}}^h \circ \dots \circ \Phi_{H_1}^h,
\end{align}
where $\Phi_{H_{i}}^h$ reads
\[
\begin{array}{l}
    x_{n+1} = x_{n} + h\frac{\partial H^{[i]}}{\partial p}(x_{n+1},p_{n}) = x_{n} + hp_{n}^{i}{\bf e}_{i}, \\
    p_{n+1} = p_{n} - h\frac{\partial H^{[i]}}{\partial x}(x_{n+1},p_{n}) = p_{n} - h\nabla \phi^{[i]}(x_{n+1}).
\end{array}
\]
Here, ${\bf e}_{j}$ is the $N$-dimensional vector with the $j$-th element being 1.

Define a discrete Lagrangian as
\begin{align}\label{eq:splitting}
    \mathbb{L}^{\text{1st}}(x_{n}, x_{n+1}, h)
    &= \frac{1}{2}\frac{\left(x_{n+1}-x_{n}\right)^\top\left(x_{n+1}-x_{n}\right)}{h^2}
    - \Big[\phi^{[1]}(x_{n+1}^1, x_{n}^2, \dots, x_{n}^{\mathrm{N}}) + \dots \notag \\
    &+ \phi^{[i]}(x_{n+1}^1, \dots, x_{n+1}^{i}, x_{n}^{i+1}, \dots, x_{n}^{\mathrm{N}}) + \dots + \phi^{[N]}(x_{n+1}^1, \dots, x_{n+1}^{\mathrm{N}})\Big].
\end{align}
Denote $\widehat{x}_{n}^{i} = x_{n} + \sum\limits_{j=1}^{i}(x_{n+1}^j - x_{n}^j) {\bf e}_{j}$, $i=1, \dots, \mathrm{N}$. It follows from \eqref{eq:splitting} that
\begin{align}\label{eq:splitting1}
    \mathbb{L}^{\text{1st}}(x_{n}, x_{n+1}, h)
    = \frac{1}{2}\frac{\left(x_{n+1}-x_{n}\right)^\top\left(x_{n+1}-x_{n}\right)}{h^2}
    - \sum_{i=1}^{\mathrm{N}} \phi^{[i]}(\widehat{x}_{n}^{i}).
\end{align}		
With the discrete Lagrangian \eqref{eq:splitting1}, we state the following theorem.

\begin{theorem}\label{Thm:Splitting}
The discrete Lagrangian is of order 1. Moreover, the discrete Euler--Lagrange equations induced by \eqref{eq:splitting1} are given by
\[
\begin{aligned}
    &\frac{x_{n+1}^1 - 2 x_{n}^1 + x_{n-1}^1}{h^2} = - \frac{\partial}{\partial x_1} \left( \sum_{j=1}^{\mathrm{N}} \phi^{[j]}(\widehat{x}_{n-1}^j) \right), \\
    &\frac{x_{n+1}^{i} - 2 x_{n}^{i} + x_{n-1}^{i}}{h^2} = - \frac{\partial}{\partial x_{i}} \left( \sum_{j=1}^{i-1} \phi^{[j]}(\widehat{x}_{n}^j) + \sum_{j=i}^{\mathrm{N}} \phi^{[j]}(\widehat{x}_{n-1}^j) \right), \quad i = 2, \dots, \mathrm{N},
\end{aligned}
\]
and it is equivalent to the symplectic method $\Phi_{h}$ \eqref{eq:Hcomposition}.
\end{theorem}

\begin{proof}
From the Taylor expansion, we have 
\[
\frac{x(t_{n+1}) - x(t_{n})}{h} = \dot{x}(t) + \mathcal{O}(h)
\]
and 
\[
\phi^{[i]}(\widehat{x}_{n}^{i}) = \phi^{[i]}(x(t)) + \mathcal{O}(h)
\]
on the interval $[t_{n}, t_{n+1}]$. Substituting these into \eqref{eq:splitting1}, we find that the discrete Lagrangian has first order accuracy.

Introduce the discrete Legendre transform
\begin{align}
    &p_{n} = -h\frac{\partial \mathbb{L}^{\text{1st}}(x_{n}, x_{n+1})}{\partial x_{n}} = \frac{x_{n+1} - x_{n}}{h} + h \sum\limits_{i=2}^{\mathrm{N}} \frac{\partial}{\partial x_{i}} \sum\limits_{j=1}^{i-1} \phi^{[j]}(\widehat{x}_{n}^j) {\bf e}_{i}, \label{eq:Thm_splitting_eq1} \\
    &p_{n+1} = h\frac{\partial \mathbb{L}^{\text{1st}}(x_{n}, x_{n+1})}{\partial x_{n+1}} = \frac{x_{n+1} - x_{n}}{h} + h \sum\limits_{i=1}^{\mathrm{N}} \frac{\partial}{\partial x_{i}} \sum\limits_{j=i}^{\mathrm{N}} \phi^{[j]}(\widehat{x}_{n}^j) {\bf e}_{i}. \label{eq:Thm_splitting_eq2}
\end{align}
From \eqref{eq:Thm_splitting_eq1} and \eqref{eq:Thm_splitting_eq2}, we can derive
\begin{align}
    &x_{n+1} = x_{n} + h p_{n} - h^2 \sum\limits_{i=2}^{\mathrm{N}} \frac{\partial}{\partial x_{i}} \sum\limits_{j=1}^{i-1} \phi^{[j]}(\widehat{x}_{n}^j) {\bf e}_{i}, \label{eq:Phi1}\\
    &p_{n+1} = p_{n} - h \sum\limits_{i=1}^{\mathrm{N}} \nabla \phi^{[i]}(\widehat{x}_{n}^{i}). \label{eq:Phi2}
\end{align}
This is exactly the composition method $\Phi_{h}$ \eqref{eq:Hcomposition}.
\end{proof}

We construct the discrete Lagrangian of high order by a given discrete Lagrangian and its adjoint. For instance, for system \eqref{eq:Eq1}, the discrete Lagrangian $\mathbb{L}^{\text{1st}}$ \eqref{eq:splitting1} provides the discrete Lagrangians of second order:
\begin{equation}\label{eq:VI2_1}
    \mathbb{L}^{\text{2nd}}\left(x_{n}, x_{n+1}, h\right)
    = \frac{1}{2} \mathbb{L}^{\text{1st}}\left(x_{n}, x_{n+1/2}, h/2\right)
    + \frac{1}{2} \mathbb{L}^{*}\left(x_{n+1/2}, x_{n+1}, h/2\right),
\end{equation}
or
\begin{equation}\label{eq:VI2_2}
    \mathbb{L}^{\text{2nd}}\left(x_{n}, x_{n+1}, h\right)
    = \frac{1}{2} \mathbb{L}^{*}\left(x_{n}, x_{n+1/2}, h/2\right)
    + \frac{1}{2} \mathbb{L}^{\text{1st}}\left(x_{n+1/2}, x_{n+1}, h/2\right),
\end{equation}
where $\mathbb{L}^*$ is the adjoint of $\mathbb{L}^{\text{1st}}$. We also have the following theorem.

\begin{theorem}\label{Thm:Hcompsition_2}
    The variational integrator derived from the adjoint discrete Lagrangian $\mathbb{L}^*$ is equivalent to the adjoint of $\Phi_h$ defined in \eqref{eq:Hcomposition}, denoted by $\Phi_h^*$. For the second order discrete Lagrangians, the integrator from \eqref{eq:VI2_1} is equivalent to
    \[
    \Phi_h^{\text{2nd}} = \Phi_{h/2}^* \circ \Phi_{h/2},
    \]
    while the integrator from \eqref{eq:VI2_2} is equivalent to
    \[
    {\Phi_h^{\text{2nd}}}^* = \Phi_{h/2} \circ \Phi_{h/2}^*.
    \]
\end{theorem}

\begin{proof}
By Definition \ref{def:disLeg}, the momentum variables are given by:
\begin{align*}
    p_n
    &= - h \frac{\partial \mathbb{L}^*(x_n, x_{n+1}, h)}{\partial x_n}
    = - h \frac{\partial \mathbb{L}^{\text{1st}}(x_{n+1}, x_n, -h)}{\partial x_n},\\
    p_{n+1}
    &= h \frac{\partial \mathbb{L}^*(x_n, x_{n+1}, h)}{\partial x_{n+1}}
    = h \frac{\partial \mathbb{L}^{\text{1st}}(x_{n+1}, x_n, -h)}{\partial x_{n+1}}.
\end{align*}
These expressions coincide with the discrete Legendre transforms for $\mathbb{L}(x_{n+1}, x_n, -h)$, proving that the variational integrator derived from $\mathbb{L}^*$ is equivalent to the adjoint map $\Phi_h^*$.

We prove the equivalence of the variational integrator in \eqref{eq:VI2_1} and $\Phi_h^{\text{2nd}}$. The map $\Phi_h^{\text{2nd}}:(x_n, p_n) \rightarrow (x_{n+1}, p_{n+1})$ satisfies:
\begin{align}
    p_n
    &= - h \frac{\partial \mathbb{L}(x_n, x_{n+1}, h)}{\partial x_n}
    = - \frac{h}{2} \frac{\partial \mathbb{L}^{\text{1st}}(x_n, x_{n+1/2}, h)}{\partial x_n}, \label{eq:Hcom2_eq1}\\
    p_{n+1}
    &= h \frac{\partial \mathbb{L}(x_n, x_{n+1}, h)}{\partial x_{n+1}}
    = \frac{h}{2} \frac{\partial \mathbb{L}^*(x_n, x_{n+1/2}, h)}{\partial x_{n+1}}. \label{eq:Hcom2_eq2}
\end{align}
Introducing
\[
    p_{n+1/2} = \frac{h}{2} \frac{\partial \mathbb{L}^{\text{1st}}}{\partial x_{n+1/2}}(x_n, x_{n+1/2}),
\]
and using the discrete Euler--Lagrange equations, we find:
\begin{equation}\label{eq:Hcom2_eq3}
    p_{n+1/2}
    = \frac{h}{2} \frac{\partial \mathbb{L}^{\text{1st}}}{\partial x_{n+1/2}}(x_n, x_{n+1/2})
    = - \frac{h}{2} \frac{\partial \mathbb{L}^*}{\partial x_{n+1/2}}(x_{n+1/2}, x_{n+1}).
\end{equation}
Combining \eqref{eq:Hcom2_eq1}--\eqref{eq:Hcom2_eq3} proves that $\Phi_h^{\text{2nd}} = \Phi_{h/2}^* \circ \Phi_{h/2}$.
\end{proof}
	
\noindent {\bf The relativistic Kepler problem.} To simulate the motion of a slowly accelerated charged particle in an attractive Coulomb field (as in electrodynamics), we consider the Kepler system with relativistic corrections, described by
\begin{equation}\label{eq:relaKepler}
    \frac{d}{dt}\left(\gamma(\dot{x})\dot{x}\right) = -\frac{x}{|x|^3},
\end{equation}
where $\gamma(\dot{x}) = \frac{1}{\sqrt{1-\frac{|\dot{x}|^2}{c^2}}}$ is the Lorentz factor, and $c$ is the speed of light. It is known that there exists a potential function $\phi(x) = - \frac{1}{|x|}$ such that $- \nabla \phi(x) = \frac{x}{|x|^3}$. In the non-relativistic limit, where $|\dot{x}| \ll c$ (i.e., $\frac{|\dot{x}|}{c} \rightarrow 0$), the system \eqref{eq:relaKepler} reduces to the classical Kepler problem \eqref{eq:Kepler}.
Since
\[
    \frac{\partial \gamma(\dot{x})}{\partial \dot{x}_{i}}\dot{x}_{j} = \frac{\partial \gamma(\dot{x})}{\partial \dot{x}_{j}}\dot{x}_{i} = \frac{\dot{x}_{i}\dot{x}_{j}}{c^2}\left(1 - \left(\frac{\dot{x}}{c}\right)^{2}\right)^{-3/2},
\]
it is verified that the conditions (\ref{eq:varcond1}a) and (\ref{eq:varcond1}d) hold. Therefore, the relativistic Kepler system \eqref{eq:relaKepler} can be written in a variational formulation. The Lagrangian is taken as
\begin{align}\label{eq:Lagrangian_relaKepler}
    L(x, \dot{x}) = -\frac{c^2}{\gamma(\dot{x})} + \frac{1}{|x|}.
\end{align}

We reformulate the equation \eqref{eq:relaKepler} by introducing the proper time $\tau$ and the momentum $u = \gamma(\dot{x}) \dot{x}$. In the new coordinates, the relativistic Kepler system becomes 
\begin{equation}\label{eq:relaKepler_pro}
    \begin{alignedat}{2}
        \frac{dt}{d\tau} &= \gamma, & \quad
        \frac{dx}{d\tau} &= u, \\
        \frac{d\gamma}{d\tau} &= - \nabla \phi(x)\cdot u, & \quad
        \frac{du}{d\tau} &= - \gamma \nabla \phi(x).
    \end{alignedat}
\end{equation}
In a compact form, it is 
\begin{equation}\label{eq:relaKepler_promat}
    {\bf M} \ddot{\mathbf{x}} = F({\bf x}) \dot{\mathbf{x}},
\end{equation}
where $\cdot$ denotes the differentiation with respect to $\tau$, ${\bf x} = (t, x^{\top})^{\top}$, 
$
{\bf M} = \begin{bmatrix}
    -1 & \\
    & \mathrm{I}_{\mathrm{N}}
\end{bmatrix}
$, and 
$
F({\bf x}) = \begin{bmatrix}
    0 & \nabla^{\top} \phi(x) \\
    - \nabla \phi(x) & \mathbf{0}
\end{bmatrix}.
$
According to Theorem \ref{Thm:M=M(dot x)}, the system \eqref{eq:relaKepler_promat} is derived from a variational principle. The Lagrangian is defined as 
\begin{equation}\label{eq:Lagrangian_relaKepler_pro}
    L({\bf x}, \dot{\mathbf{x}}) = \frac{1}{2}\dot{\mathbf{x}}^{\top} {\bf M} \dot{\mathbf{x}} + A^{\top}({\bf x}) \dot{\mathbf{x}},
\end{equation}
with $A({\bf x}) = (-\phi(x), \mathbf{0}_{1 \times \mathrm{N}})$.

Define $p = \frac{\partial L}{\partial \dot{\mathbf{x}}}({\bf x}, \dot{\mathbf{x}}) = {\bf M}\dot{\mathbf{x}} + A({\bf x})$, the system \eqref{eq:relaKepler_pro} takes a standard Hamiltonian form 
\[
\dot{y} = J^{-1} \nabla H(y),
\]
where $y = ({\bf x}^{\top}, p^{\top})^{\top}$. Here, 
$J = \left[\begin{array}{cc}  & \mathrm{I}_{\mathrm{N}+1} \\ -\mathrm{I}_{\mathrm{N}+1} & \end{array}\right]$ and the Hamiltonian is expressed as 
$H({\bf x}, p) = \frac{1}{2} (p - A({\bf x}))^{\top} {\bf M} (p - A({\bf x}))$. Introducing 
$
\mathbf{v} = \frac{\partial L}{\partial \dot{\mathbf{x}}}({\bf x}, \dot{\mathbf{x}})
- A({\bf x})
= {\bf M} \dot{\mathbf{x}} = \left(-\gamma, u^{\top}\right)^{\top},
$ 
the system \eqref{eq:Lagrangian_relaKepler_pro} is rewritten as a $K$-symplectic system \cite{feng1995collected, feng2010symplectic}
\begin{equation}\label{eq:relaKepler_Ksymmat}
    \dot{{\bf z}} = K^{-1}({\bf z}) \nabla H({\bf z}),
\end{equation}
where ${\bf z} = ({\bf x}^{\top}, \mathbf{v}^{\top})^{\top}$, 
$
K({\bf z}) = \left[
\begin{array}{cc}
    F({\bf x}) & -\mathrm{I}_{\mathrm{N+1}} \\ \mathrm{I}_{\mathrm{N+1}} & \mathbf{0}
\end{array}
\right],
$ 
and the Hamiltonian $H$ is 
$
H(\mathbf{x},\mathbf{v}) = \frac{1}{2} \mathbf{v}^{\top} {\bf M} \mathbf{v}.
$
	
Next, we construct the numerical methods for \eqref{eq:relaKepler_Ksymmat}. The Hamiltonian $H({\bf x}, \mathbf{v})$ can be split into $\mathrm{N}+1$ parts:
\[
H({\bf x}, \mathbf{v}) = - \frac{1}{2}\gamma^2 + \frac{1}{2}\sum\limits_{i=1}^{\mathrm{N}} u_{i}^2 = H_{t} + \sum\limits_{i=1}^{\mathrm{N}} H_{i}.
\]
Substituting this into \eqref{eq:relaKepler_Ksymmat} yields $\mathrm{N}+1$ subsystems, each subsystem remains $K$-symplectic and admits an explicit solution.

The subsystem generated by $H_{t}$ is
\begin{equation}\label{eq:H_t}	
	\begin{alignedat}{2}
		\frac{dt}{d\tau} &= \gamma, & \quad \frac{dx}{d\tau} &= 0, \\
		\frac{d\gamma}{d\tau} &= 0, & \quad \frac{du}{d\tau} &= -\gamma \nabla \phi(x).
	\end{alignedat}
\end{equation}
The exact solution for this subsystem is
\begin{equation}\label{eq:H_t_exact}	
	\begin{alignedat}{2}
		t(\tau + h) &= t + h\gamma(\tau), & \quad x(\tau+h) &= x(\tau), \\
		\gamma(\tau + h) &= \gamma(\tau), & \quad u(\tau+h) &= u(\tau) - h\gamma(\tau)\nabla \phi(x(\tau)).
	\end{alignedat}
\end{equation}
For a given $i$ ($i = 1, \dots, N$), the subsystem corresponding to $H_{i}$ is
\begin{equation}\label{eq:H_{v_i}}	
	\begin{alignedat}{2}
		\frac{dt}{d\tau} &= 0, & \quad \frac{dx}{d\tau} &= u {\bf e}_{i}, \\
		\frac{d\gamma}{d\tau} &= - \frac{\partial \phi(x)}{\partial x_{i}} u_{i}, & \quad \frac{du}{d\tau} &= 0.
	\end{alignedat}
\end{equation}
The exact solution for this subsystem is
\begin{equation}\label{eq:H_{v_i}_exact}	
	\begin{alignedat}{2}
		t(\tau + h) &= t(\tau), & \quad x(\tau+h) &= x(\tau) + hu(\tau) {\bf e}_{i}, \\
		\gamma(\tau + h) &= \gamma(\tau) - \int_{x_{i}(\tau)}^{x_{i}(\tau) + hu_{i}(\tau)} \frac{\partial \phi(x)}{\partial x_{i}} dx_{i}, & \quad u(\tau+h) &= u(\tau).
	\end{alignedat}
\end{equation}
The integral in \eqref{eq:H_{v_i}_exact} can be computed explicitly as
\[
\int_{x_{i}(\tau)}^{x_{i}(\tau) + hu_{i}(\tau)} \frac{\partial \phi(x)}{\partial x_{i}} \, dx_{i} = \phi(x)\Big|_{x(\tau)}^{x(\tau) + hu(\tau) {\bf e}_{i}}.
\]
Thus, the exact solution \eqref{eq:H_{v_i}_exact} for this subsystem to $H_{i}$ reduces to 
\begin{equation}\label{eq:H_{v_i}_exact1}	
	\begin{alignedat}{2}
		t(\tau + h) &= t(\tau), & \quad x(\tau+h) &= x(\tau) + hu(\tau) {\bf e}_{i}, \\
		\gamma(\tau + h) &= \gamma(\tau) - \phi(x)\Big|_{x(\tau)}^{x(\tau) + hu(\tau) {\bf e}_{i}}, & \quad u(\tau+h) &= u(\tau).
	\end{alignedat}
\end{equation}
Composing \eqref{eq:H_t_exact} and \eqref{eq:H_{v_i}_exact1}, we obtain a $K$-symplectic method for \eqref{eq:relaKepler_pro}
\begin{equation}\label{eq:relaKepler_Ksym}
    \Phi_{h} = \phi_{H_{n}}^{h} \circ \dots \circ \phi_{H_1}^{h} \circ \phi_{H_t}^{h}.
\end{equation}

\begin{theorem}\label{Thm:Ksym = VI}			
 Define the discrete Lagrangian for system \eqref{eq:relaKepler_pro} as
\begin{align}\label{eq:vainteg_{n}}
    \mathbb{L}^{\text{1st}}\left({\bf x}_{n}, {\bf x}_{n+1}, h\right)
    = \frac{1}{2}\frac{\left({\bf x}_{n+1} - {\bf x}_{n}\right)^\top \left({\bf x}_{n+1} - {\bf x}_{n}\right)}{h^2} + {\bf A}^{\top}({\bf x}_{n})\frac{\left({\bf x}_{n+1} - {\bf x}_{n}\right)}{h}.
\end{align}
Then the discrete Euler-Lagrange equation yields the following variational integrator:
\begin{equation}\label{eq:DEL for relaKepler}	
    \begin{aligned}
        &\frac{t_{n+1} - 2 t_{n} + t_{n-1}}{h^2} = - \frac{\phi(x_{n}) - \phi(x_{n-1})}{h}, \\
        &\frac{x_{n+1} - 2 x_{n} + x_{n-1}}{h^2} = - \frac{t_{n+1} - t_{n}}{h}\nabla \phi(x_{n}).
    \end{aligned}
\end{equation}
Also, we prove that the method \eqref{eq:DEL for relaKepler} is equivalent to the $K$-symplectic method \eqref{eq:relaKepler_Ksym}.
\end{theorem}

\begin{proof}
    From the discrete Euler-Lagrange equation \eqref{eq:DEL for relaKepler} with \(\mathbb{L}^{\text{1st}}\), we introduce the following transformations:
    \[
    {\bf v}_{n} = - h \frac{\partial \mathbb{L}^{\text{1st}}\left({\bf x}_{n}, {\bf x}_{n+1} \right)}{\partial {\bf x}_{n}} - {\bf A}({\bf x}_{n})
    \]
    and
    \[
    {\bf v}_{n + 1} = h \frac{\partial \mathbb{L}^{\text{1st}}\left({\bf x}_{n}, {\bf x}_{n+1} \right)}{\partial {\bf x}_{n+1}} - {\bf A}({\bf x}_{n+1}).
    \]
    Denote \({\bf v}_{n} = \left( -\gamma_{n}, u_{n} \right)\), from the above equality we have
    \begin{align}\label{eq:Thm_Ksym=VI_eqs}
        \begin{alignedat}{3}
            \gamma_{n} &= \frac{t_{n+1} - t_{n}}{h}, & \quad
            u_{n} &= \frac{x_{n+1} - x_{n}}{h} + \left(t_{n+1} - t_{n}\right)\nabla \phi(x_{n}), \\
            \gamma_{n + 1} &= \frac{t_{n+1} - t_{n}}{h} + \left(\phi(x_{n+1}) - \phi(x_{n})\right), & \quad
            u_{n+1} &= \frac{x_{n+1} - x_{n}}{h}.
        \end{alignedat}
    \end{align}
    Reformulating \eqref{eq:Thm_Ksym=VI_eqs} leads to 
    \begin{align}
        \begin{alignedat}{3}
            t_{n+1} &= t_{n} + h\gamma_{n}, & \quad
            x_{n+1} &= x_{n} + h u_{n} - h\left(t_{n+1} - t_{n}\right)\nabla \phi(x_{n}), \\
            \gamma_{n+1} &= \gamma_{n} + \left( \phi(x_{n+1}) - \phi(x_{n}) \right), & \quad
            u_{n+1} &= u_{n} - \left(t_{n+1} - t_{n}\right)\nabla \phi(x_{n}),
        \end{alignedat}
    \end{align}
    which are exactly the $K$-symplectic methods \eqref{eq:relaKepler_Ksym}.
\end{proof}

From \eqref{eq:vainteg_{n}}, it is known that the adjoint discrete Lagrangian is
\begin{align*}
	\mathbb{L}^*\left(\mathbf{x}_{n}, \mathbf{x}_{n+1}, h\right)
	= \mathbb{L}^{\text{1st}}\left(\mathbf{x}_{n+1}, \mathbf{x}_{n},-h\right)
	= \frac{1}{2}\frac{\left({\bf x}_{n+1} - {\bf x}_{n}\right)^\top \left({\bf x}_{n+1} - {\bf x}_{n}\right)}{h^2} 
	+ {\bf A}^{\top}({\bf x}_{n})\frac{\left({\bf x}_{n+1} - {\bf x}_{n}\right)}{h}.
\end{align*}
Similarly to Theorem \ref{Thm:Ksym = VI}, the following theorem holds.
\begin{theorem}\label{Thm:Ksym = VI*}
	The variational integrator deduced from $\mathbb{L}^*\left(\mathbf{x}_{n}, \mathbf{x}_{n+1},h\right)$
	is equal to
	\[
	\Phi_{h}^* = \phi_{H_t}^{h} \circ \phi_{H_1}^{h} \circ \dots \circ \phi_{H_{n}}^{h}.
	\]
\end{theorem}
Furthermore, with the second order discrete Lagrangian
\begin{align}\label{eq:vainteg_2a}
	\mathbb{L}^{\text{2nd}}(\mathbf{x}_{n},\mathbf{x}_{n+1}, h)
	= \frac{1}{2} \mathbb{L}^*\left(\mathbf{x}_{n}, \mathbf{x}_{n + 1/2}, \frac{h}{2}\right)
	+ \frac{1}{2} \mathbb{L}^{\text{1st}}\left(\mathbf{x}_{n + 1/2}, \mathbf{x}_{n+1}, \frac{h}{2}\right),
\end{align}
and
\begin{align}\label{eq:vainteg_2b}
	\mathbb{L}^{\text{2nd}}(\mathbf{x}_{n},\mathbf{x}_{n+1}, h)
	= \frac{1}{2} \mathbb{L}^{\text{1st}}\left(\mathbf{x}_{n}, \mathbf{x}_{n + 1/2}, \frac{h}{2}\right)
	+ \frac{1}{2} \mathbb{L}^*\left(\mathbf{x}_{n + 1/2}, \mathbf{x}_{n+1}, \frac{h}{2}\right),
\end{align}
we have the following theorem.
\begin{theorem}\label{Thm:Ksym = VI2}
	The variational integrators deduced from \eqref{eq:vainteg_2a} and \eqref{eq:vainteg_2b} are equivalent to
	\[
	\Phi^{\text{2nd}}_{h}=\phi_{H_{n}}^{h/2} \circ \dots \circ \phi_{H_1}^{h/2} \circ \phi_{H_t}^{h} \circ \phi_{H_1}^{h/2} \circ \dots \circ \phi_{H_{n}}^{h/2},
	\]
	and
	\[
	\Phi^{\text{2nd}}_{h}=\phi_{H_t}^{h/2} \circ \phi_{H_1}^{h/2} \circ \dots \circ  \phi_{H_{n}}^{h} \circ \dots \circ \phi_{H_1}^{h/2} \circ \phi_{H_t}^{h/2},
	\]
	respectively.
\end{theorem}

\section{Modified equations for variational integrator}
In this section, we analyze the numerical conservative behaviour of variational integrators constructed above. First, we introduce the concept of modified equations introduced in \cite{vermeeren2017modified} for the second order differential equations in the form of
\begin{align}\label{eq:DE}
    \ddot{x} = f(x).
\end{align}
\begin{definition}
    Denote $x_{n+1} = \Psi(x_{n}, x_{n-1}, h)$ as a numerical method for the second order differential system \eqref{eq:DE}. The formal differential equations
    \begin{align}\label{eq:MDE}
        \ddot{\widetilde x} = \widetilde{f}(\widetilde x, \dot{\widetilde x}, h) = f(\widetilde x) + h f^{[1]}(\widetilde x, \dot{\widetilde x}) + h^2 f^{[2]}(\widetilde x, \dot{\widetilde x}) + \cdots
    \end{align}
    are called its modified equations, if the exact solution ${\widetilde x}(t)$ of \eqref{eq:MDE} coincides with the numerical solution $x_{n}$ at $t=nh$.
\end{definition}
	
Expand the difference $x(t+h) - \Psi(x(t), x(t-h), h)$ around $x(t)$, which is expressed as
\begin{align}\label{eq:NM_Taylor}
    x(t+h) - \Psi(x(t), x(t-h), h) = \ddot{x} - \widetilde{f}(x, \dot x, h) + h g_1[x] + h^2 g_{2}[x] + \cdots,
\end{align}
where $g_{i}[x], i=1,2,\cdots$ are functions depending on $x$ and its derivatives.
Introduce the intermediate variable $v = \dot{x}$, it follows from \eqref{eq:MDE} that
\begin{align*}
    \ddot{x} &= \dot{v} = \widetilde{f}(x, v, h) = f(x) + h f^{[1]}(x, v) + h^2 f^{[2]}(x, v) + \cdots,\\
    x^{(3)}
    & = \ddot{v} = \partial_x f(x) v + h\left(\partial_x f^{[1]}(x,v) v + \partial_v f^{[1]}(x, v) f(x) + \partial_v f(x) f^{[1]}(x, v)\right) + \cdots,\\
    & \cdots\cdots
\end{align*}
Substituting the above expressions into \eqref{eq:NM_Taylor} and comparing the coefficients in $h$, we get
\begin{align*}
    h^1: \quad &f^{[1]}(x, v) + g_1(x, v, f, \partial_x f v + \partial_v f f), \\
    h^2: \quad &f^{[2]}(x, v) + g_2(x, v, f, \partial_x f v + \partial_v f f)
    + \partial_{\dot{v}} g_1(x, v, f, \partial_x f v + \partial_v f f) f^{[1]} \\
    &+ \partial_{\ddot{v}} g_1(x, v, f, \partial_x f v + \partial_v f f)\left(\partial_x f^{[1]}(x, v) v + \partial_v f^{[1]}(x, v) f(x) + \partial_v f(x) f^{[1]}(x, v)\right) + \cdots, \\
    &\cdots \cdots
\end{align*}
Therefore, we have the following theorem.
\begin{theorem}
    The coefficients $f^{[k]}$ in the modified equations \eqref{eq:MDE} are uniquely determined recursively by $f^{[1]}, \ldots$, $f^{[k-1]}, g_1, \ldots, g_{k-1}$, and their derivatives.
\end{theorem}
	
Generally, the modified equation is a divergent formal series. However, for some special cases, we can obtain a convergent modified equation. In the following example, we show that for a linear second order differential equation, the corresponding modified equation is convergent.

Consider the linear second order differential equation
\begin{equation}\label{eq:linearsystem}
    \ddot{x} = - \lambda x,\quad \lambda>0.
\end{equation}
For the linear system \eqref{eq:linearsystem}, we apply the following numerical method:
\begin{align}\label{eq:NumMet_linear}
    \frac{x_{n+1} - 2x_{n} + x_{n-1}}{h^2} = -\lambda x_{n}.
\end{align}
It is easy to verify that the difference equation \eqref{eq:NumMet_linear} has a solution $x_{n}$, which is written in the following form with initial values $x_{0} = x(t_{0})$ and $x_{1}=x(t_{0}+h)$:
\[
    x_{n} = a e^{-2in\theta} + b e^{2in\theta}, \quad \theta = \arcsin\left(\frac{\sqrt{\lambda}h^2}{2}\right),
\]
where $a$ and $b$ are determined by $x_{0}$ and $x_{1}$.
Denote $x(t) = a e^{-2it\theta/h} + b e^{2it\theta/h}$, and from \cite{vermeeren2018modified}, it is known that
\begin{align}\label{eq:chebyshev}
    x(t - j h) -2x(t) + x(t + j h)
    = 2 \left(T_j\left(1 - \frac{\lambda h^2}{2}\right)-1\right) x(t)
    = (-1)^j \lambda^{j}h^{2j}x(t) + P_{2j-2}(h) x(t),
\end{align}
where \(T_j\) is the $j$-th order Chebyshev polynomial of the first kind, and $P_{2j-2}$ is the polynomial of degree $2j-2$.
Expand the left hand side of \eqref{eq:chebyshev} around $t$ for $j=1,2,\cdots$, we have
\begin{align}\label{eq:taylor}
    x(t - j h) - 2 x(t) + x(t + j h) = (jh)^2 \ddot{x}(t) + \mathcal{O}(h^4).
\end{align}
For $j=1, 2, \cdots, k$, we get
\begin{align}\label{eq:expand}
    \begin{split}
        x(t - h) - 2 x(t) + x(t + h) 
        &=  h^2 \ddot{x}(t)+ \frac{2 h^{4}}{4!} x^{(4)}(t) + \cdots +\frac{2 h^{2k}}{(2k)!} x^{(2k)}(t) + \mathcal{O}(h^{2k+2}),\\
        x(t - 2 h) - 2 x(t) + x(t + 2 h) 
        &=  (2h)^2 \ddot{x}(t) + \frac{2 (2h)^{4}}{4!} x^{(4)}(t) + \cdots + \frac{2 (2h)^{2k}}{(2k)!} x^{(2k)}(t)+\mathcal{O}(h^{2k+2}),\\
        &\cdots \cdots \\ 
        x(t - k h) - 2 x(t) + x(t + k h)
        &= (kh)^2 \ddot{x}(t)+ \frac{2 (kh)^{4}}{4!} x^{(4)}(t) + \cdots + \frac{2 (kh)^{2k}}{(2k)!} x^{(2k)}(t)+\mathcal{O}(h^{2k+2}).
    \end{split}
\end{align}
Denote
\[
    A= \begin{bmatrix}
    1 & 1 & \cdots & 1 \\
    2^2 & 2^4 & \cdots & 2^{2k} \\
    \vdots & \vdots & \ddots & \vdots \\
    k^2 & k^4 & \cdots & k^{2k}
    \end{bmatrix}, \quad
    B= \begin{bmatrix}
    x(t-h) - 2x(t) + x(t+h) & 1 & \cdots & 1 \\
    \vdots & \vdots & \ddots & \vdots \\
    x(t-kh) - 2x(t) + x(t+kh) & k^4 & \cdots &k^{2k}
    \end{bmatrix}.
\]
Using Cramer's rule, from \eqref{eq:expand} we obtain
\begin{equation}\label{eq:coe_of_ddot}
    \ddot{x}(t) = h^{-2} \frac{\det(B)}{\det(A)} + \mathcal{O}(h^{2k}).
\end{equation}
Substituting \eqref{eq:chebyshev} into \(B\), from \eqref{eq:coe_of_ddot} the coefficient of \(h^{2k-2}\) is calculated as \(-\frac{2(k-1)!^2 \lambda^k}{(2k)!} x(t)\).
This implies the modified equation of \eqref{eq:NumMet_linear} has the closed form
\[
    \ddot{x} = -\sum_{k=1}^\infty \frac{2(k-1)!^2}{(2k)!} h^{2k-2} \lambda^k x.
\]
	
As follows, we consider the system \eqref{eq:DE} with $f(x)=-\nabla \phi(x)$. It is known that the system  has  a variational formulation.  However, the numerical discretizations  of  system are usually not variational. By the way of modified equations, the numerical discretization is variational if and only its corresponding modified equation has a variational formulation. According to Theorem \ref{Thm:M}, this implies that the coefficients of the modified equations $f^{[k]}, k=1,2,\cdots$ in \eqref{eq:MDE} must satisfy
\begin{align*}
    &\frac{\partial f^{[k]}_{i}}{\partial \dot x_{j}} + \frac{\partial f^{[k]}_{j}}{\partial \dot x_{i}}=0,\\
    &\frac{\partial f^{[k]}_{i}}{\partial x_{j}}-\frac{\partial f^{[k]}_{j}}{\partial x_{i}}+\frac{d}{dt}\frac{\partial f^{[k]}_{i}}{\partial \dot x_{j}}=0.
\end{align*}
The corresponding Lagrangian of the modified equations is shown further as the modified Lagrangian denoted by $\mathbb{L}_{\text{mod}}$. To analyze the numerical behavior of  variational integrator, in the following  we introduce the concept of \emph{modified Lagrangian} \cite{vermeeren2017modified}.
	
\begin{definition}
The formal function given by
\begin{align}\label{eq:modL}
    \mathbb{L}_{\text{mod}}(x, \dot{x}, h) = L(x, \dot{x}) + h L^{[1]}(x, \dot{x}) + h^2 L^{[2]}(x, \dot{x}) + \cdots
\end{align}
is said to be the modified Lagrangian if the continuous solution of the corresponding Euler–Lagrange equations
\begin{align}\label{eq:mod_EL}
    \frac{\partial \mathbb{L}_{\text{mod}}}{\partial x}(x, \dot{x}, h) - \frac{d}{dt} \frac{\partial \mathbb{L}_{\text{mod}}}{\partial \dot{x}}(x, \dot{x}, h) = 0
\end{align}
and the numerical solution of the variational integrator satisfies \( x(jh) = x_{j} \).
\end{definition}

An alternative definition  of modified Lagrangian  presented  in \cite{vermeeren2017modified} is as follows
\begin{equation}\label{eq:Lmod}
    \int_{t_{n}}^{t_{n+1}} \mathbb{L}_{\text{mod}}(x(t), \dot{x}(t), h) dt = h \mathbb{L}(x_{n}, x_{n+1}, h) \approx \int_{t_{n}}^{t_{n+1}} L\left(x(t), \dot{x}(t)\right) dt,
\end{equation}
where $t_{n} = t_{0} + nh$. In fact, by differentiating \eqref{eq:Lmod} on both sides w.r.t to $x_{n}$, we obtain
\begin{align*}
    \partial_1 \mathbb{L}(x_{n}, x_{n+1}, h) + \partial_2 \mathbb{L}(x_{n-1}, x_{n}, h)
    &=  \int_{t_{n-1}}^{t_{n+1}} \left( \frac{\partial \mathbb{L}_{\text{mod}}}{\partial x} \cdot \frac{\partial x(t)}{\partial x_{n}}
    + \frac{\partial \mathbb{L}_{\text{mod}}}{\partial \dot{x}} \cdot \frac{\partial \dot{x}(t)}{\partial x_{n}} \right) dt\\
    &=  \int_{t_{n-1}}^{t_{n+1}} \left( \frac{\partial \mathbb{L}_{\text{mod}}}{\partial x}
    - \frac{d}{dt} \frac{\partial \mathbb{L}_{\text{mod}}}{\partial \dot{x}} \right) \cdot \frac{\partial x(t)}{\partial x_{n}} dt
    + \left[ \frac{\partial x(t)}{\partial x_{n}} \right]\Bigg|_{t_{n-1}}^{t_{n+1}}.
\end{align*}
Expand the discrete Lagrangian $\mathbb{L}(x(t), x(t+h), h)$  at  $x(t+\frac{h}{2})$, it reads
\[
    \mathbb{L}(x(t), x(t+h), h) 
    = \mathbb{L}\left([x(t+\frac{h}{2})], h\right) 
    :=  L(x, \dot{x}) 
    + \sum_{k=1}^{\infty} h^{k} L_{k}\left([x]\right),
\]
where $x$ denote $x\left(t+\frac{h}{2}\right)$, $[x]$ contains $x$ and its derivatives. According to Proposition 6 in \cite{vermeeren2017modified}, the term $L_{k}$ only depends on $x$ and its $j$ order derivatives $x^{(j)}, j\leq k+1$.
	
By Lemma 1 in \cite{vermeeren2017modified}, the discrete action functional over the interval $[t_{0}, t_{N}]$ can be expressed as
\begin{align}\label{eq:action}
    \begin{split}
        \mathbb{S}(\left\{x_{n}\right\}_{n=0}^{\mathrm{N}})
        &= \sum_{k=0}^{\mathrm{N}-1} h\mathbb{L}(x(t_{0}+kh), x(t_{0}+(k+1)h), h) \\
        &= \int_{t_{0}}^{t_{N}} \sum_{k=0}^{\infty} \left(2^{1-2k} - 1\right) \frac{h^{2k} B_{2k}}{(2k)!} \frac{d^{2k}}{dt^{2k}} \mathbb{L}([x(t)], h) dt,
    \end{split}
\end{align}
where $B_{k}$ are Bernoulli numbers with $B_{0} = 1$, $B_{2} = \frac{1}{6}$ and
$
    B_{n} = \lim\limits_{x \rightarrow0} \frac{d^{n}}{dx^{n}}\left(\frac{x}{e^x-1}\right)
$. From \eqref{eq:modL}, it follows that
\begin{equation}\label{eq:Lmod([x])}
    \int_{t_{0}}^{t_{N}} \mathbb{L}_{\text{mod}}(x, \dot{x}, h) dt
    = \int_{t_{0}}^{t_{N}} dt \,\left( \mathbb{L}([x], h)
    - \frac{h^2}{24} \frac{d^2}{dt^2} \mathbb{L}([x], h)
    + \frac{7h^4}{5760} \frac{d^4}{dt^4} \mathbb{L}([x], h)
    + \cdots \right).
\end{equation}
Solve $\ddot{x} = \widetilde{f}(x, \dot{x}, h)$ implicitly from the modified Euler--Lagrange equation \eqref{eq:mod_EL}
\[
    \frac{\partial L_{\text{mod}}}{\partial x}(x, \dot{x}, h) - \frac{d}{dt}\frac{\partial L_{\text{mod}}}{\partial \dot{x}}(x, \dot{x}, h) = 0,
\]
and substitute it into \eqref{eq:Lmod([x])}, eliminating $\ddot{x}$ and its higher-order derivatives in the right-hand side of \eqref{eq:Lmod([x])}.
Comparing the coefficients of the term $h^k$ on both sides of \eqref{eq:Lmod([x])}, we obtain
\begin{align*}
    h^1: \quad &L^{[1]}(x, \dot{x}) = L_1(x, \dot{x}),\\
    h^2: \quad &L^{[2]}(x, \dot{x}) = L_2(x, \dot{x}, f,  \partial_{x}f\dot{x} + \partial_{\dot{x}}ff) 
    - \frac{1}{24}\left(
    L_{xx}\left(\dot{x}, \dot{x}\right)
    + 2L_{x\dot{x}}\left(\dot{x}, f\right)
    + L_{\dot{x}\dot{x}}\left(f, f\right)
    + L_{x}f
    + L_{\dot{x}}\left(\partial_{x}f\dot{x} + \partial_{\dot{x}}f\,f \right)\right),\\
    & \cdots \cdots
\end{align*}
where elementary differentials are used to represent the higher-order derivatives, as described in \cite{hairer2006geometric}.
Recursively, we can get the modified Lagrangian in formal series corresponding to the time step $h$.
    
For the Kepler problem \eqref{eq:Kepler}, we split $\phi(x)$ as two parts $\phi = \phi^{[1]} + \phi^{[2]}$. Then, we can construct a discrete Lagrangian of first order as
\begin{align}\label{eq:La1}
    \mathbb{L}(x_{n}, x_{n+1}, h) = \frac{1}{2} \frac{(x_{n+1} - x_{n})^2}{h^2} - \left[ \phi^{[1]}(x_{n+1}^1, x_{n}^2) + \phi^{[2]}(x_{n+1}^1, x_{n+1}^2) \right].
\end{align}
The variational integrator that follows from the discrete Euler-Lagrange equation is
	\begin{align*}
		\frac{x_{n+1} - 2 x_{n} + x_{n-1}}{h^2} = -
		\left[
		\begin{array}{l}
			\frac{\partial }{\partial {x_1}} \left( \phi^{[1]}(x_{n}^1, x_{n-1}^2) + \phi^{[2]}(x_{n}^1, x_{n}^2) \right) \\
			\frac{\partial }{\partial {x_2}} \left( \phi^{[1]}(x_{n+1}^1, x_{n}^2) + \phi^{[2]}(x_{n}^1, x_{n}^2) \right)
		\end{array}
		\right].
	\end{align*}
The discrete Lagrangian of second order is constructed as
	\begin{align}\label{eq:La2}
		\mathbb{L}^{\text{2nd}}(x_{n}, x_{n+1}, h) = \frac{1}{2}\mathbb{L}^*(x_{n}, x_{n+1/2}, h/2) + \frac{1}{2}\mathbb{L}^{\text{1st}}(x_{n+1/2}, x_{n+1}, h/2),
	\end{align}
where $\mathbb{L}^{*}$ is the adjoint Lagrangian of $\mathbb{L}^{\text{1st}}$. By the discrete Euler-Lagrange equations, the corresponding second-order variational integrator is expressed as
\begin{align*}
    \begin{split}
        \frac{x_{n+1/2} - 2 x_{n} + x_{n-1/2}}{(h/2)^2} &= -\left[\begin{array}{c}
	0 \\
        \frac{\partial }{\partial {x_2}} \left( \phi^{[1]}(x_{n+1/2}^1, x_{n}^2) + \phi^{[1]}(x_{n-1/2}^1, x_{n}^2) \right)
	\end{array}\right]\\
        \frac{x_{n+1} - 2 x_{n+1/2} + x_{n}}{(h/2)^2} &= - \left[\begin{array}{c}
        \frac{\partial }{\partial {x_1}} \left( \phi^{[1]}(x_{n+1/2}^1, x_{n}^2) + 2 \phi^{[2]}(x_{n+1/2}^1, x_{n+1/2}^2) + \phi^{[1]}(x_{n+1/2}^1, x_{n+1}^2) \right) \\
        \frac{\partial }{\partial {x_2}} \left( 2 \phi^{[2]}(x_{n+1/2}^1, x_{n+1/2}^2) \right)
	\end{array}\right].
    \end{split}
\end{align*}	
	
According to the above process, we obtain the modified Lagrangian for \eqref{eq:La1} and \eqref{eq:La2} in the following proposition.

\begin{theorem}\label{Thm:Mod_Lag}
The modified Lagrangians  for the newly constructed discrete Lagrangian of order 1 \eqref{eq:La1} and order 2 \eqref{eq:La2} are given as
\[
    \mathbb{L}_{\text{mod}}^{\text{1st}}(x, \dot{x}, h)
    = \frac{1}{2}|\dot{x}|^2 - \phi(x)
    + \frac{h}{2} \left( \frac{\partial \phi}{\partial x_1} \dot{x}_1 + \frac{\partial (\phi^{[2]} - \phi^{[1]})}{\partial x_2} \dot{x}_2 \right)
    + \mathcal{O}(h^2)
\]
and
\begin{align*}
    \mathbb{L}_{\text{mod}}^{\text{2nd}}(x, \dot{x}, h) &= \frac{1}{2}|\dot{x}|^2 - \phi(x)
    + \frac{h^2}{96} \left(
    7\left(\frac{\partial \phi}{\partial x_{1}}\right)^2
    - 5\left(\frac{\partial \phi^{[1]}}{\partial x_2}\right)^2
    + 2\frac{\partial \phi^{[1]}}{\partial x_{2}}\frac{\partial \phi^{[2]}}{\partial x_{2}}
    + 7\left(\frac{\partial \phi^{[2]}}{\partial x_2}\right)^2
    \right)\\
    &+ \frac{h^2}{24} \left(
    - 2\frac{\partial^2 \phi}{\partial x_{1}^2} \dot{x}_{1}^2
    + 2\frac{\partial^2 \phi^{[1]}}{\partial x_1 \partial x_2} \dot{x}_{1} \dot{x}_{2}
    +  \frac{\partial^2 \phi^{[1]}}{\partial x_{2}^2} \dot{x}_{2}^2
    - 4\frac{\partial^2 \phi^{[2]}}{\partial x_{1} \partial x_{2}} \dot{x}_{1} \dot{x}_{2}
    - 2\frac{\partial^2 \phi^{[2]}}{\partial x_{2}^2} \dot{x}_2^2
    \right)
    + \mathcal{O}(h^4).
\end{align*}
\end{theorem}
	
\begin{proof}
Without loss of generality, we compute the modified Lagrangian for the second-order discrete Lagrangian \eqref{eq:La2}. Expanding \eqref{eq:La2} around \(x(t+h/2)\), we derive
\begin{align*}
    \mathbb{L}^{\text{2nd}} \left([x], h\right)
    &= \frac{1}{4} \left[ \frac{(x(t+h/2) - x(t))^2}{(h/2)^2} + \frac{(x(t+h) - x(t+h/2))^2}{(h/2)^2} \right]
    - \frac{1}{2} \Big[\phi^{[1]}\left(x_{1}\left(t\right), x_{2}\left(t+h/2\right)\right)\\
    &+ \phi^{[1]}\left(x_{1}(t+h), x_{2}(t+h/2)\right)
    + \phi^{[2]}\left(x_{1}(t), x_{2}(t)\right)
    + \phi^{[2]}\left(x_{1}(t+h), x_{2}(t+h)\right) \Big]\\
    &= \frac{1}{2}|\dot{x}|^2 - \phi(x)
    + \frac{h^2}{96} \left(
    3|\ddot{x}|^2
    + 4 \langle \dot{x}, x^{(3)} \rangle
    - 12 \frac{\partial \phi^{[1]}}{\partial x_{2}} \ddot{x}_{2}
    - 12 \frac{\partial^2 \phi^{[1]}}{\partial x_2^2} \dot{x}_{2}^2
    - 12 \frac{\partial \phi^{[2]}}{\partial x} \ddot{x}
    - 12 \frac{\partial^2 \phi^{[2]}}{\partial x^2} \dot{x}^2
    \right)\\
    &+ \mathcal{O}(h^4).
\end{align*}
Thus, from \eqref{eq:Lmod([x])}, we have
\begin{align}\label{eq:Lmod[x]_2nd}
    \begin{split}
        \mathbb{L}^{\text{2nd}}_{\text{mod}}(x,\dot{x}, h)
        =& \mathbb{L}^{\text{2nd}}([x,h])
        - \frac{h^2}{24} \frac{d^2}{dt^2} \mathbb{L}^{\text{2nd}}([x],h) + \mathcal{O}(h^4) \\
        =& \frac{1}{2}|\dot{x}|^2 - \phi(x) + \frac{h^2}{96} \Big(
        - |\ddot{x}|^2
        - 8\frac{\partial \phi}{\partial x_1} \ddot{x}_1
        + 4\frac{\partial \phi^{[1]}}{\partial x_{2}} \ddot{x}_{2}
        - 8\frac{\partial \phi^{[2]}}{\partial x_{2}} \ddot{x}_{2}
        - 8\frac{\partial^2 \phi}{\partial x_1^2} \dot{x}_1^2  \\
        &+ 8\frac{\partial^2 \phi^{[1]}}{\partial x_1 \partial x_2} \dot{x}_1 \dot{x}_2
        + 4\frac{\partial^2 \phi^{[1]}}{\partial x_{2}^2} \dot{x}_{2}^2
        - 16\frac{\partial^2 \phi^{[2]}}{\partial x_{1}\partial x_{2}} \dot{x}_1 \dot{x}_2
        - 8\frac{\partial^2 \phi^{[2]}}{\partial x_{2}^2} \dot{x}_{2}^2 \Big)
        + \mathcal{O}(h^4).
    \end{split}
\end{align}
By the Euler-Lagrange equation with \(\mathbb{L}^{\text{2nd}}_{\text{mod}}\) in expression \eqref{eq:Lmod[x]_2nd},
$$
    \frac{\partial \mathbb{L}^{\text{2nd}}_{\text{mod}}(x,\dot{x}, h)}{\partial x} - \frac{d}{dt}\frac{\partial \mathbb{L}^{\text{2nd}}_{\text{mod}}(x,\dot{x}, h)}{\partial \dot{x}} = 0,
$$
we obtain the modified equation in series: \(\ddot{x} = - \nabla \phi(x) + \mathcal{O}(h^2)\). Substituting \(\ddot{x}\) into \eqref{eq:modL} gives the results recursively.
\end{proof}
	
The Kepler problem \eqref{eq:Kepler} has three conserved quantities: energy, angular momentum, and the Laplace--Runge--Lenz (LRL) vector. By means of backward error analysis \cite{hairer2006geometric}, it is known that the variational integrators presented above can ensure that energy and angular momentum remain bounded over long times. In the following, we investigate how the newly derived variational integrators \eqref{eq:La1} and \eqref{eq:La2} preserve the LRL vector through their modified Lagrangian. When the initial energy $H_0 < 0$, the solution of the Kepler problem is elliptic, such as the motion of one astronomical body around another. In this case, the LRL vector determines the orientation and shape of the elliptical orbit. Specifically, the magnitude $\left| \mathbf{A} \right| = \sqrt{A_1^2 + A_2^2}$ is proportional to the eccentricity of the orbit, while the vector lies along the major axis of the orbit. The angle relative to the first coordinate axis is given by $\omega = \arctan\left(\frac{A_2}{A_1}\right)$, providing a precise description of the orientation.

For the perturbed Kepler problem, assume that the corresponding Lagrangian is given by $\widetilde{L} = L(x, \dot{x}) + \varepsilon \overline{L}(x, \dot{x})$, and we state the following theorem.
	
\begin{theorem}\label{Thm:omega&e}
If the major axis of the orbit is $\mathcal{O}(\varepsilon)$-close to the $x_2$-axis, the variations in the eccentricity and the angle of the LRL vector over a period are given by
\begin{equation}\label{eq:ecc}
    \Delta \left| \mathbf{A} \right| = - \varepsilon T \left[\langle \operatorname{EL}(\overline{L}), v_{A_2} \rangle \right] + \mathcal{O}(\varepsilon^2)
\end{equation}
and the change in the angle $\omega$ is given by
\begin{equation}\label{eq:angle}
    \Delta \omega = \frac{\varepsilon T}{e} \left[\langle \operatorname{EL}(\overline{L}), v_{A_1} \rangle \right] + \mathcal{O}(\varepsilon^2),
\end{equation}
where $T$ is the period of the solution to the Kepler problem, and $v_{A_{i}}$ are the characteristics of the LRL vector presented in Proposition \ref{GVF_Kepler}. Here $[\quad ]$ denotes the time average over one period $T$.
\end{theorem}
	
\begin{proof}
We choose coordinates such that $A_1 = \mathcal{O}(\varepsilon)$ and $A_2 \approx e$. Using the Euler-Lagrange operator, the perturbed Kepler problem has the form $\operatorname{EL}(\widetilde{L}) = 0$. By Theorem \ref{PerGVF}, we have
\begin{align}
    \frac{d A_{i}}{dt} = - \varepsilon \left\langle \operatorname{EL}(\overline{L}), v_{A_{i}} \right\rangle.
\end{align}
To calculate the variations in eccentricity and the angle of the LRL vector, we obtain
\begin{align*}
    \frac{d}{dt}\left| {\bf A} \right| &= \frac{d}{dt} \left( \sqrt{A_1^2 + A_2^2} \right) = \frac{1}{\sqrt{A_1^2 + A_2^2}} \left( A_1 \frac{d A_1}{dt} + A_2 \frac{d A_2}{dt} \right),\\
    \dot{\omega}
    &= \frac{d}{dt} \left( \arctan \frac{A_2}{A_1} \right) = \frac{1}{A_1^2 + A_2^2} \left( A_1 \frac{d A_2}{dt} - A_2 \frac{d A_1}{dt} \right).
\end{align*}
Therefore, we obtain
\begin{align*}
    \frac{d}{dt}\left|{\bf A}\right|
    &= - \varepsilon \left\langle \operatorname{EL}(\overline{L}), v_{A_{2}} \right\rangle
    + \mathcal{O}(\varepsilon^2),\\
    \dot{\omega}
    &= \frac{\varepsilon}{e} \left\langle \operatorname{EL}(\overline{L}), v_{A_{1}} \right\rangle
    + \mathcal{O}(\varepsilon^2).
\end{align*}
This completes the proof by integrating both sides with respect to time $t$ over one period.
\end{proof}
	
\begin{theorem}\label{Thm:Modified_Lagrangian}
For the symplectic Euler method, the errors in eccentricity and the angle of the LRL vector over one period are of order $2$.
For the newly derived discrete Lagrangian of order 1 in \eqref{eq:La1}, if $\phi^{[1]} = \phi^{[2]}$ is taken, then the errors in eccentricity and the angle of the LRL vector are also of order 2.
\end{theorem}
	
\begin{proof}
The symplectic Euler method is associated with a variational integrator, and its modified Lagrangian can be computed as
\[
    \mathbb{L}_{\text{mod}}^{\text{symE}} = \frac{1}{2} \left| \dot{x} \right|^2 - \phi(x) - \frac{h}{2} (\dot{x}^{\top} \nabla \phi(x) + \mathcal{O}(h)).
\]
Setting $\varepsilon = \frac{h}{2}$ and $\overline{L} = - \dot{x}^{\top} \nabla \phi(x) + \mathcal{O}(\varepsilon)$ in Theorem \ref{Thm:omega&e}, we observe that $\operatorname{EL}(\dot{x}^{\top} \nabla \phi(x)) = 0$. Consequently,
\[
    \Delta \left| \mathbf{A} \right| = \mathcal{O}(\varepsilon^2), \quad \Delta \omega = \mathcal{O}(\varepsilon^2).
\]
Thus, for the symplectic Euler method, the variations in both the eccentricity and the angle of the LRL vector over one period are of order 2.

For the newly derived discrete Lagrangian of order 1, it follows from Theorem \ref{Thm:Modified_Lagrangian} that its modified Lagrangian is given by
\[
    \mathbb{L}^{\text{1st}}_{\text{mod}} = \frac{1}{2} \left| \dot{x} \right|^2 - \phi(x) + \frac{h}{2} \frac{x_1 \dot{x}_1}{\left| x \right|^3} + \mathcal{O}(h^2).
\]
Letting $\varepsilon = \frac{h}{2}$ and $\overline{L} = \frac{x_1 \dot{x}_1}{\left| x \right|^3} + \mathcal{O}(\varepsilon)$, we have
\[
    \operatorname{EL}\left( \frac{x_1 \dot{x}_1}{\left| x \right|^3} \right)
    = \frac{3}{|x|^5} \left[ - x_1 x_2 \dot{x}_2, x_1 x_2 \dot{x}_1 \right]^\top.
\]
Next, we compute the inner products between $\operatorname{EL}\left(\frac{x_1 \dot{x}_1}{|x|^3}\right)$ and the components of the LRL vector:
\begin{align*}
    \left\langle \operatorname{EL}\left(\frac{x_1 \dot{x}_1}{|x|^3}\right), v_{A_1} \right\rangle
    &= 6 \frac{H x_1 x_2^2}{|x|^5}
    - 6 \frac{A_2 x_1 x_2}{|x|^5}, \\
    \left\langle \operatorname{EL}\left(\frac{x_1 \dot{x}_1}{|x|^3}\right), v_{A_2} \right\rangle
    &= -6 \frac{H x_1^2 x_2}{|x|^5}
    + 6 \frac{A_1 x_1 x_2}{|x|^5}.
\end{align*}
Transforming to polar coordinates with $x_1 = -r \sin(\theta)$ and $x_2 = r \cos(\theta)$, where $\theta = 0$ aligns with the positive $x_2$ axis, we apply Kepler's laws (see Proposition \ref{Thm:Kepler_3}) to compute the following integral averages over one period:
\[
    \left[ \frac{x_1 x_2}{|x|^5} \right]
    = - \frac{1}{T} \int_0^T \frac{\cos(\theta) \sin(\theta)}{r^3} \, dt
    = - \frac{1}{2\pi b^3} \int_{0}^{2\pi} (1 + e\cos{\theta})\cos(\theta) \sin(\theta) \, d\theta
    = 0.
\]
Similarly, we compute:
\[
    \left[ \frac{x_{1}^2 x_{2}}{|x|^5} \right] = \frac{1}{2\pi ab} \int_{0}^{2\pi} \cos(\theta) \sin^2(\theta) \, d\theta = 0, \quad
    \left[ \frac{x_1 x_2^2}{|x|^5} \right] = - \frac{1}{2\pi ab} \int_{0}^{2\pi} \cos^2(\theta) \sin(\theta) \, d\theta = 0.
\]
Thus, for the discrete Lagrangian of order 1 \eqref{eq:La1}, it follows that the variations in both the eccentricity and the angle of the LRL vector over one period are of order 2.
\end{proof}

In \cite{vermeeren2017modified}, the modified Lagrangian for the St\"{o}rmer--Verlet method is derived as
\[
    \mathbb{L}_{\text{mod}}^{\text{SV}}(x, \dot{x}, h) = \frac{1}{2} |\dot{x}|^2 - \phi(x) + \frac{h^2}{24} \left( 
    \frac{1}{|x|^4}  
    - 2 \frac{|\dot{x}|^2}{|x|^3} 
    + 6 \frac{(x^{\top} \dot{x})^2}{|x|^5}
    \right)
    + \mathcal{O}(h^4).
\]
Additionally, the convergence rate of the angle is investigated in \cite{vermeeren2018numerical}. The following theorem establishes the superconvergence of the error in eccentricity for the St\"{o}rmer--Verlet method.

\begin{theorem}
    The error in eccentricity for the St\"{o}rmer--Verlet method is of order 4.
\end{theorem}

\begin{proof}
    Let $\varepsilon = \frac{h^2}{24}$ and $\overline{L} = \frac{1}{|x|^4}  
    - 2 \frac{|\dot{x}|^2}{|x|^3} 
    + 6 \frac{(x^{\top} \dot{x})^2}{|x|^5} + \mathcal{O}(\varepsilon^2)$. Following the calculation in \cite{vermeeren2017modified}, we obtain the expression
    \[
    \left[\left\langle \operatorname{EL}\left(\frac{1}{|x|^4}  
    - 2 \frac{|\dot{x}|^2}{|x|^3} 
    + 6 \frac{(x^{\top} \dot{x})^2}{|x|^5} \right), v_{A_2} \right\rangle\right]
    = \left[60 \frac{x_1}{|x|^6} + 48H \frac{x_1}{|x|^5} - 30m^2 \frac{x_1}{|x|^7}\right]m + \mathcal{O}\left(\varepsilon^2\right).
    \]
    We now introduce polar coordinates, with $x_1 = -r \sin(\theta)$ and $x_2 = r \cos(\theta)$, and compute the following integral:
    \[
    \left[ \frac{x_1}{|x|^k} \right]
    = - \frac{1}{T} \int_0^T \frac{\sin(\theta)}{r^{k-1}} \, dt
    = - \frac{a^{k-4}}{2\pi b^{2k-5}} \int_0^{2\pi} (1 + e \cos(\theta))^{k-3} \sin(\theta) \, d\theta
    = 0,\quad k=5,\ 6,\ 7.
    \]
    Thus, for the St\"{o}rmer--Verlet method, the variations in eccentricity over one period are of order 4.
\end{proof}

\section{Numerical Experiments}
In this section, we present the numerical results for applying the newly constructed first order and second order discrete Lagrangians \eqref{eq:La1} and \eqref{eq:La2}. We denote them by VI-1 and VI-2, respectively. In comparison, we also use the symplectic Euler (sym-Euler) and the St\"{o}rmer--Verlet (SV) methods. To apply the variational integrators, which are multi-step methods in $x$, we need to specify $x_{1} \approx x(h)$. To do this, we can select $x_{1}$ such that $v_{0} = - h\partial_{1}\mathbb{L}(x_{0}, x_{1})$ by the discrete Legendre transform \ref{def:disLeg}.

Take the initial conditions as $(x_{0}, v_{0}) = \left(1-e, 0, 0, \sqrt{\frac{1+e}{1-e}}\right)$, where $e$ is the eccentricity. In this experiment, we set $e = 0.6$. In Fig.~\ref{orbit_Kepler}, we demonstrate the numerical orbits over the time interval $[0, 200]$ with a time step of $h = 0.05$. As seen in Fig.~\ref{orbit_Kepler}, the orbits calculated by the symplectic Euler and St\"{o}rmer--Verlet methods exhibit significant clockwise rotation. In contrast, the variational integrators of first order and second order \eqref{eq:La1} and \eqref{eq:La2} simulate the elliptic orbit well. The first order variational integrator shows only a slight counter-clockwise precession, while the second order integrator exhibits a very tiny phase shift.
\begin{figure}[H]
	\centering
	\subfigure[]{\includegraphics[scale=0.38]{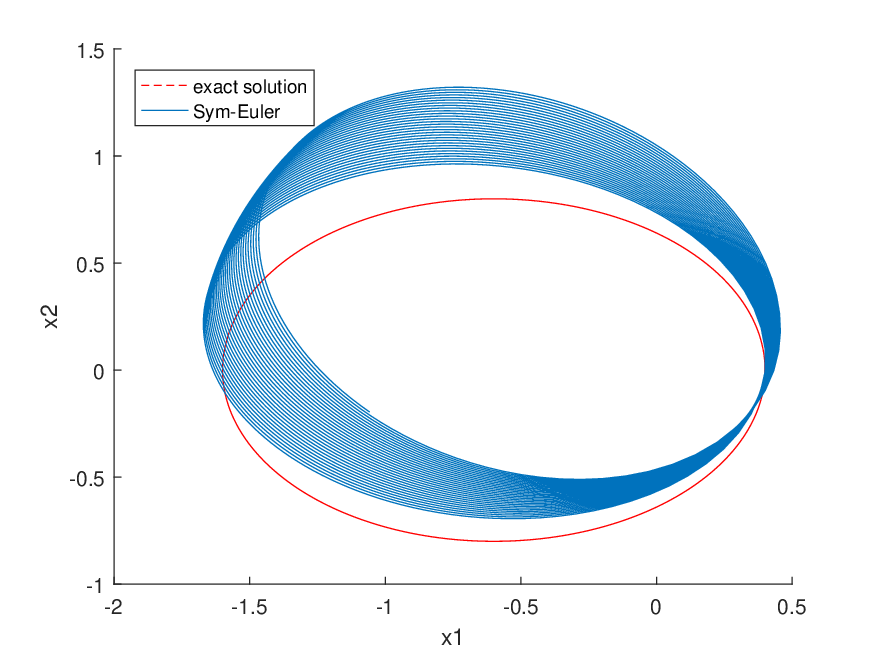}}
	\quad
	\subfigure[]{\includegraphics[scale=0.38]{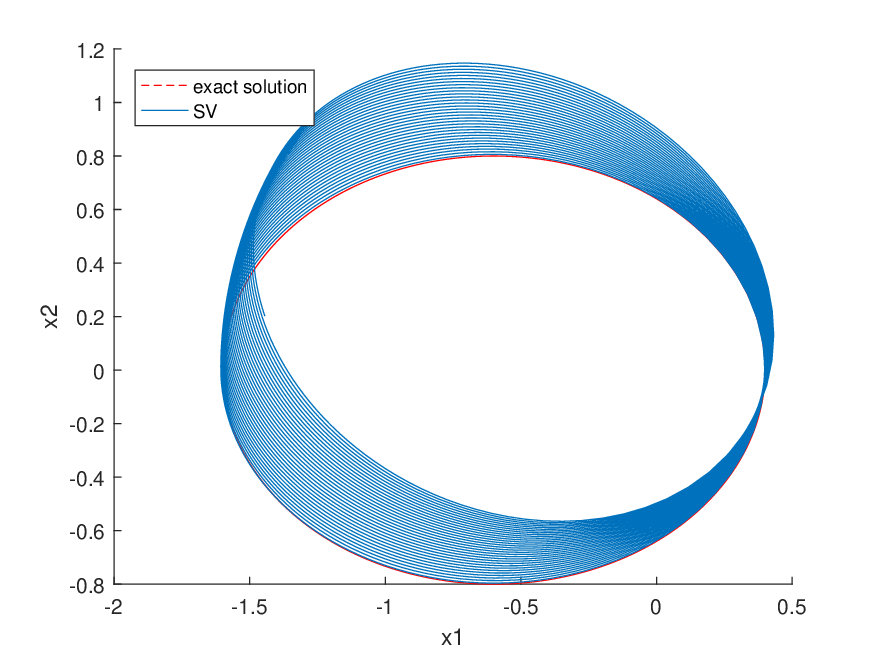}}
	\quad
	\subfigure[]{\includegraphics[scale=0.38]{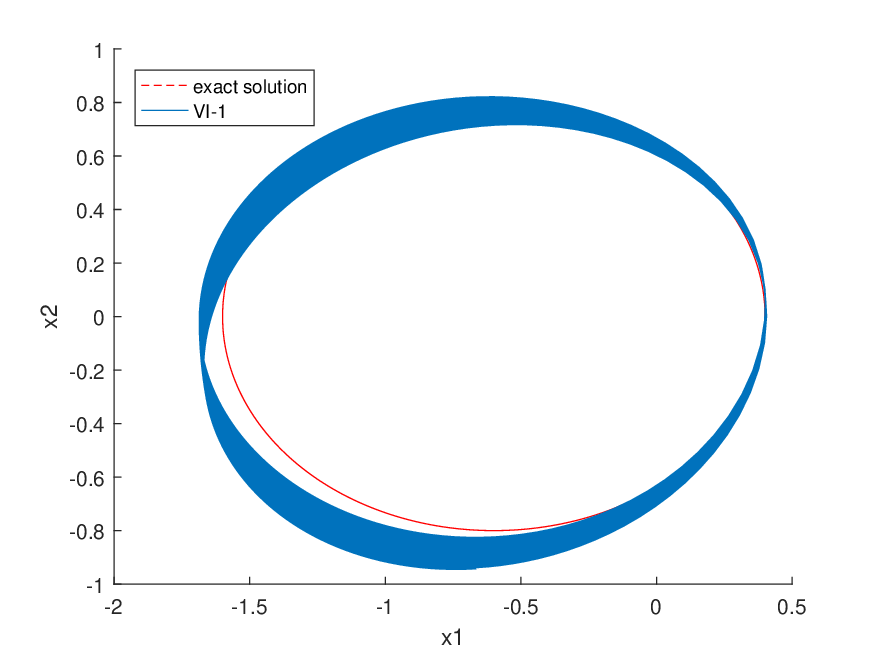}}
	\quad
	\subfigure[]{\includegraphics[scale=0.38]{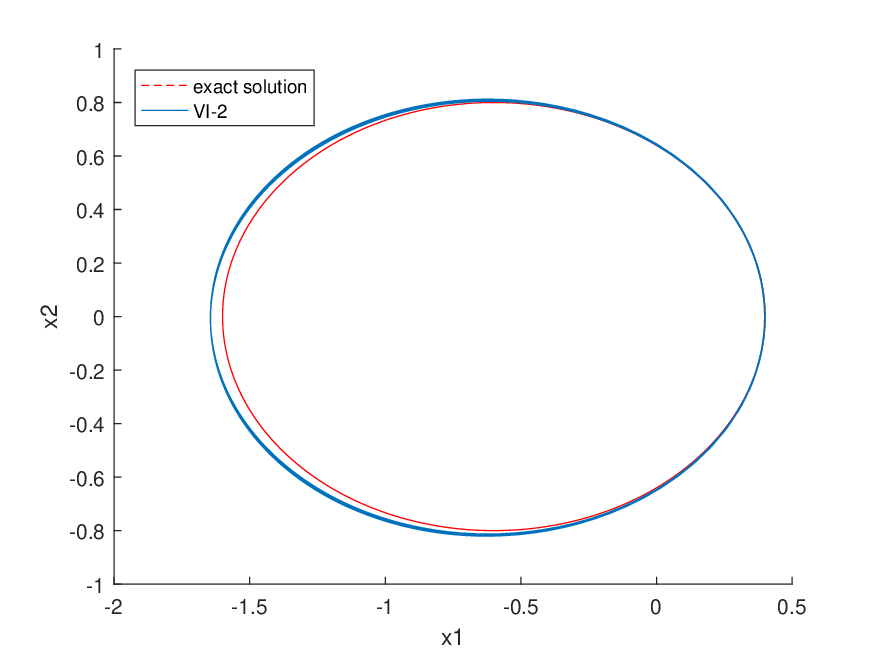}}
	\caption{Numerical solutions with time step $h = 0.05$ over 4000 time steps. Red dots represent the exact solution.}
	\label{orbit_Kepler}
\end{figure}

In this section, in addition to the simulations for numerical orbits, we also plot the errors of conserving quantities: energy, angular momentum, and the Laplace--Runge--Lenz (LRL) vector. Using the initial values $(-3, 0, 0, 0.45)$ and a time step of $h = 0.05$, all numerical computations are made over $10^5$ steps of integration.

As shown in Fig.~\ref{energy_Kepler}, all four numerical methods show bounded energy errors during the simulation. Since the St\"{o}rmer--Verlet method and the second-order variational integrator are both of second order, they exhibit more accurate energy conservation in comparison to the first-order variational integrator and the symplectic Euler method.

\begin{figure}[htbp]
    \centering
    \includegraphics[width=0.65\textwidth]{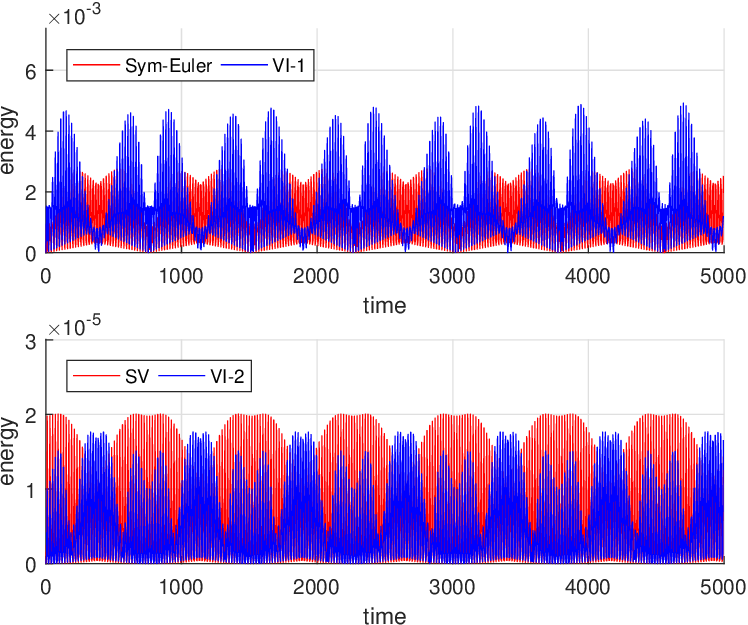}
    \caption{Numerical error of energy. Top plot: The red line shows the error for the symplectic Euler method, and the blue line shows the error for the first order variational integrator. Bottom plot: The red line shows the error for the St\"{o}rmer--Verlet method, and the blue line shows the error for the second order variational integrator.}
    \label{energy_Kepler}
\end{figure}

For the Kepler problem, as the angular momentum is quadratic, it can be preserved exactly by the symplectic Euler method and the St\"{o}rmer--Verlet method. Although the newly constructed discrete Lagrangians cannot exactly preserve the angular momentum, the error in angular momentum remains bounded during the long-term simulation. In Fig.~\ref{ecc_Kepler} and Fig.~\ref{angle_Kepler}, we present the errors of the LRL vector for the four numerical methods. Fig.~\ref{ecc_Kepler} shows the error of the eccentricity, represented by $\left|\mathbf{A}\right| = \sqrt{A_{1}^2 + A_{2}^2}$. From these plots, it can be observed that all four numerical methods keep the eccentricity errors bounded over the simulation time. In comparison with the symplectic Euler method and the St\"{o}rmer--Verlet method, the newly constructed variational integrators result in smaller eccentricity errors. In Fig.~\ref{angle_Kepler}, the error in the rotation angle is shown for the different numerical methods. Both the first order and second order variational integrators show much smaller errors compared to the symplectic Euler and St\"{o}rmer--Verlet methods. Among them, the second order variational integrator performs the best, with the smallest errors over time.

\begin{figure}[htbp]
    \centering
    \subfigure{\includegraphics[width=0.7\textwidth]{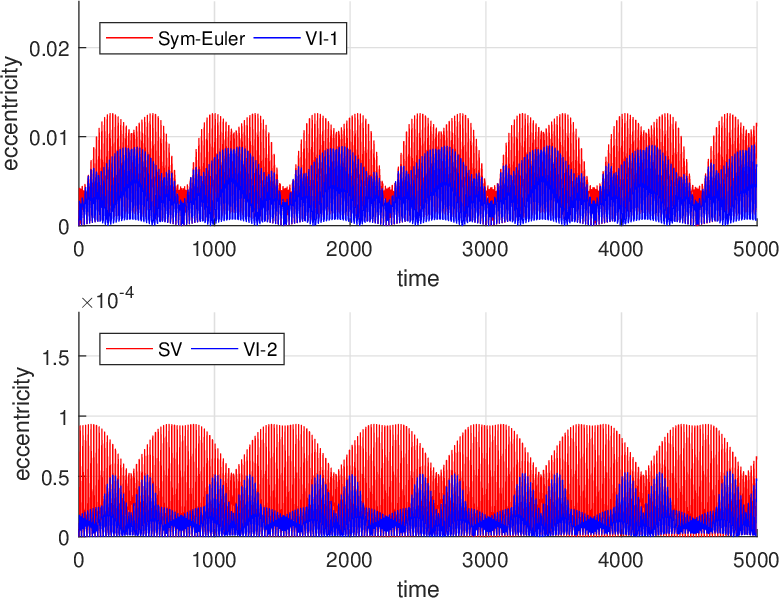}}
    \caption{Error of eccentricity. Top plot: The red line shows the error for the symplectic Euler method, and the blue line shows the error for the first order variational integrator. Bottom plot: The red line shows the error for the St\"{o}rmer--Verlet method, and the blue line shows the error for the second order variational integrator.}
    \label{ecc_Kepler}
\end{figure}

\begin{figure}[htbp]
    \centering
    \subfigure{\includegraphics[width=0.7\textwidth]{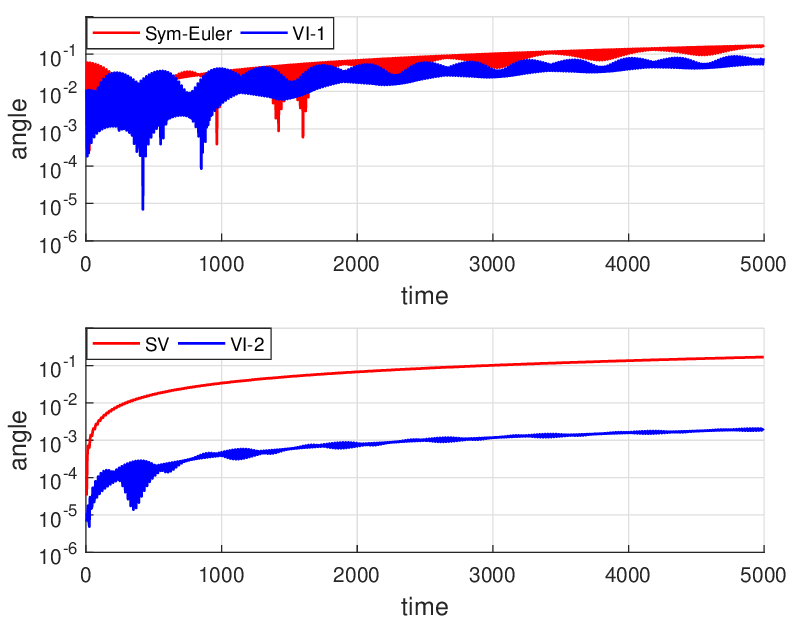}}
    \caption{Error of rotation angle. Top plot: The red line shows the error for the symplectic Euler method, and the blue line shows the error for the first order variational integrator. Bottom plot: The red line shows the error for the St\"{o}rmer--Verlet method, and the blue line shows the error for the second order variational integrator.}
    \label{angle_Kepler}
\end{figure}

In Fig.~\ref{rate_Kepler}, we show the convergence rates of the eccentricity and angle of the LRL vector using the four integration methods over one period of Keplerian motion. To estimate the convergence order, we start with $h = 0.5$. The numerical results are computed with time steps $h_i = 2^{-i}$, for $i=1, \dots, 6$. The results are presented in a log-log plot.
For the eccentricity error shown in Fig.~\ref{rate_Kepler}(a), both the symplectic Euler method and the first order variational integrator exhibit a second order convergence rate, with their error curves nearly identical. The St\"{o}rmer--Verlet method, however, shows super-convergence with an accuracy of order 4. The second order variational integrator also demonstrates super-convergence, and in comparison with the St\"{o}rmer--Verlet method, it produces smaller errors.
In Fig.~\ref{rate_Kepler}(b), we plot the convergence rate for the angle error. From the figure, it can be seen that all four numerical methods have a second order convergence rate. This is consistent with the theoretical estimates in Theorem \ref{Thm:Modified_Lagrangian}. The error curves for the symplectic Euler and St\"{o}rmer--Verlet methods are nearly identical. Both the first and second order variational integrators demonstrate smaller errors in preserving the angle of the LRL vector.
Among the four numerical methods, the second order variational integrator produces the smallest angle error. 

\begin{figure}[htbp]
\centering
\subfigure[]{\includegraphics[width=0.45\textwidth]{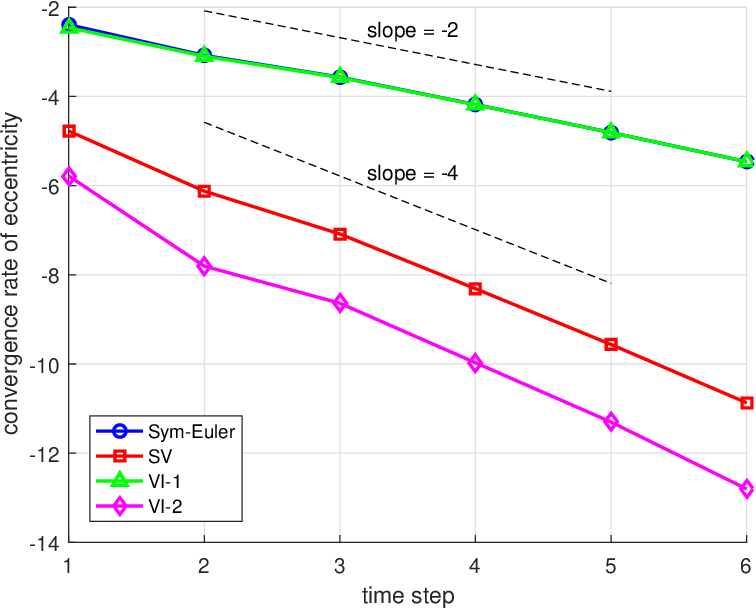}}
\quad
\subfigure[]{\includegraphics[width=0.45\textwidth]{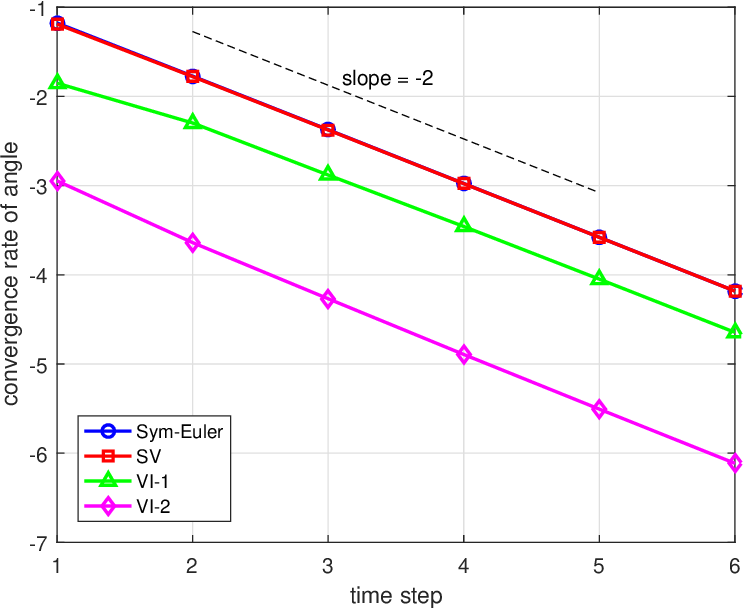}}
\caption{Convergence rates of the Laplace--Runge--Lenz (LRL) vector over one period of motion.
(a) Eccentricity error; (b) Angle error. Solid lines denote reference convergence rates.}
\label{rate_Kepler}
\end{figure}

\section{Conclusion}
In this paper, we have studied a class of second order differential systems \eqref{eq:DE_VP}, which include the classical Newtonian motion system and charged particle systems in physical plasmas. Using the inverse variational technique, we established the conditions under which this class of systems can be reformulated in a variational framework. For the Kepler problem, by splitting the potential function, we constructed variational integrators of first and second order. By introducing the appropriate Legendre transformation, we demonstrated that the variational integrators are equivalent to the composition of explicit numerical methods.
This framework enables us to perform numerical simulations of the Kepler problem more efficiently. The Kepler system, being a super-integrable system, possesses conserved quantities such as energy, angular momentum, and the Laplace--Runge--Lenz (LRL) vector. We applied the newly derived variational integrators to the two-dimensional Kepler problem. Using the method of modified equations, we established the modified Lagrangian, which is generally a formal series in step size $h$. 
Using Noether's Theorem, we presented the corresponding variational symmetries for the conservative laws of the Kepler problem. Through perturbation theory, we analyzed the errors of the first and second order variational integrators in preserving the LRL vector. The numerical errors confirm the numerical behavior observed in the experiments.
The modified Lagrangian technique provides a valuable tool for understanding the numerical behavior of variational integrators. The construction and analysis of various variational integrators for the relativistic Kepler problem and charged particle systems will be explored in future work.

\section*{Acknowledgments}
This research was supported by the National Natural Science Foundation of China No. 12271513.

\bibliographystyle{model1-num-names}

\end{document}